\documentclass[a4paper]{amsart}
\usepackage[left=3cm,right=3cm,top=3cm,bottom=3cm]{geometry}

\usepackage[utf8]{inputenc}
\usepackage{amsmath}
\usepackage{amsfonts}
\usepackage{amssymb}
\usepackage{mathrsfs, multirow}
\usepackage{sseq}

\usepackage{graphicx}
\usepackage{epigraph}
\usepackage[pagebackref]{hyperref}
\usepackage{amsthm}
\usepackage{xcolor}

\usepackage{tikz}
\usepackage{tikz-cd}
\usetikzlibrary{calc}
\usetikzlibrary{decorations.pathmorphing}

\DeclareMathOperator{\Tors}{Tors}
\DeclareMathOperator{\Hom}{Hom}
\DeclareMathOperator{\Tor}{Tor}
\DeclareMathOperator{\Ext}{Ext}
\DeclareMathOperator{\Ker}{Ker}
\DeclareMathOperator{\Coker}{Coker}
\DeclareMathOperator{\Img}{Im}
\DeclareMathOperator{\sk}{sk}
\DeclareMathOperator{\spn}{span}

\def\CC{\mathbb{C}}
\def\ZZ{\mathbb{Z}}
\def\FF{\mathbb{F}}
\def\QQ{\mathbb{Q}}
\def\RR{\mathbb{R}}

\def\Zg{\ZZ_{\geq 0}}
\def\Zm{\Zg^m}
\def\R{\mathcal{R}}
\def\Z{\mathcal{Z}}
\def\K{\mathcal{K}}

\def\ZK{\Z_\K}
\def\RK{\R_\K}
\def\OZK{\Omega\ZK}
\def\DJ{\mathrm{DJ}}
\def\caa{(\underline{CA},\underline{A})}
\def\cyy{(\underline{CY},\underline{Y})}
\def\clxx{(\underline{C\Omega X},\underline{\Omega X})}

\def\cgg{(\underline{CG},\underline{G})}
\def\ux{(\underline{X},\underline{\ast})}
\def\uxa{(\underline{X},\underline{A})}
\def\k{\mathbf{k}}

\def\H{{\widetilde{H}}}
\def\MF{\mathrm{MF}}

\newtheorem*{theorem*}{Theorem}
\newtheorem{theorem}{Theorem}[section]
\newtheorem{lemma}[theorem]{Lemma}
\newtheorem{problem}[theorem]{Problem}
\newtheorem{proposition}[theorem]{Proposition}
\newtheorem{corollary}[theorem]{Corollary}

\theoremstyle{definition}

\newtheorem{definition}[theorem]{Definition}
\newtheorem{remark}[theorem]{Remark}
\newtheorem{example}[theorem]{Example}

\numberwithin{equation}{section}

\tikzset{curve/.style={settings={#1},to path={(\tikztostart)
    .. controls ($(\tikztostart)!\pv{pos}!(\tikztotarget)!\pv{height}!270:(\tikztotarget)$)
    and ($(\tikztostart)!1-\pv{pos}!(\tikztotarget)!\pv{height}!270:(\tikztotarget)$)
    .. (\tikztotarget)\tikztonodes}},
    settings/.code={\tikzset{quiver/.cd,#1}
        \def\pv##1{\pgfkeysvalueof{/tikz/quiver/##1}}},
    quiver/.cd,pos/.initial=0.35,height/.initial=0}

\tikzset{tail reversed/.code={\pgfsetarrowsstart{tikzcd to}}}
\tikzset{2tail/.code={\pgfsetarrowsstart{Implies[reversed]}}}
\tikzset{2tail reversed/.code={\pgfsetarrowsstart{Implies}}}
\tikzset{no body/.style={/tikz/dash pattern=on 0 off 1mm}}

\graphicspath{ {./images/} }

\hypersetup{
    colorlinks=true,
    linkcolor=blue,
    citecolor=blue,
    urlcolor=blue,
    pdftitle={Anick's conjecture for polyhedral products},
    pdfauthor={Lewis Stanton and Fedor Vylegzhanin},
}

\title{Anick's conjecture for polyhedral products}
\author{Lewis Stanton*} 
\address{Department of Mathematics,
Stockholms universitet,
SE-106 91 Stockholm, Sweden}
\email{lewis.stanton@math.su.se}

\author{Fedor Vylegzhanin} 
\address{\parbox{\linewidth}{
National Research University Higher School of Economics, Moscow, Russia;\\
Steklov Mathematical Insitute of Russian Academy of Sciences, Moscow, Russia
}}
\email{vylegf@gmail.com}


\begin{document}

\begin{abstract}
We show that localised away from a finite set of primes, the loop space of a moment-angle complex is homotopy equivalent to a product of loops on spheres. As a consequence, at large primes, we give p-local loop space decompositions of quasitoric manifolds, toric orbifolds and a wide family of polyhedral products, verifying a conjecture of Anick. For such spaces, we describe p-primary part of their homotopy groups in terms of homotopy groups of spheres. We also develop a method for studying loop spaces of general polyhedral products in terms of moment-angle complexes. In particular, we describe the additive structure of loop homology of simply connected polyhedral products in terms of polynomials studied by Backelin and Berglund.
\end{abstract}
\maketitle

\section{Introduction}

Moore's conjecture asserts a deep connection between the torsion-free and torsion parts of homotopy groups. In particular, Moore conjectures that a finite, simply connected $CW$-complex has finitely many rational homotopy groups if and only if for every prime $p$, there is a finite power of $p$ which annihilates the $p$-torsion part of the homotopy groups. An approximation to Moore's conjecture is Anick's conjecture \cite{anick_loop}. This asserts that if $X$ is a finite, simply connected $CW$-complex, then localised at all but finitely many primes, the pointed loop space $\Omega X$ decomposes up to homotopy as a product of spheres, loops on spheres, and a list of well-studied indecomposable torsion spaces which were defined by Cohen, Moore and Neisendorfer. We study Anick's conjecture in the context of polyhedral products. 

Polyhedral products are subspaces of Cartesian products which unify various constructions across mathematics (see \cite{BBC20}), and the study of their pointed loop space has recently attracted much interest \cite{PT19,St24a,St25,vylegzhanin,El26}. For flag complexes, the loop homology of certain polyhedral products have close connections to right-angled Artin groups and right-angled Coxeter groups, which arise in geometric group theory \cite{GPTW16,PRah24,cartesian}. The loop homology of moment-angle complexes, a special case of a polyhedral product, is also closely related to diagonal subspace arrangements \cite{Dob09}. Moment-angle complexes are important in toric topology \cite{BP15}, since many toric manifolds and orbifolds arise as their quotients by suitable torus actions.

In this paper, we reduce the study of several homotopy-theoretic properties of the loop spaces of polyhedral products to the case of moment-angle complexes. In particular, we prove Anick's conjecture for all moment-angle complexes after localisation away from an explicit finite set of primes, and obtain similar loop space decompositions for a wide family of polyhedral products and for partial quotients of moment-angle complexes, including toric manifolds and complex toric varieties corresponding to simplicial fans. For all but finitely many primes, we express the $p$-primary parts of homotopy groups of these spaces in terms of the $p$-primary parts of the homotopy groups of spheres. Our reduction has further applications to torsion in loop homology, Steenrod algebra actions on loop cohomology, and Poincar\'e series.

To state the results more precisely, we recall the definition of the polyhedral product. Let $\K$ be a simplicial complex on the vertex set $[m]:=\{1,\dots,m\}$, and for $1 \leq i \leq m$, let $(X_i,A_i)$ be a pair of pointed $CW$-complexes. The \textit{polyhedral product associated to} $\K$ is \[\uxa^\K := \bigcup\limits_{I \in \K} \left(\prod\limits_{i=1}^m Y^I_i\right)\subseteq\prod_{i=1}^m X_i,\] where $Y^I_i = X_i$ if $i \in I$, and $Y^I_i = A_i$ if $i \notin I$. When each $(X_i,A_i) = (D^2,S^1)$, the polyhedral product is called a \textit{moment-angle complex} and is denoted $\ZK$. This is closely associated to the \textit{Davis--Januszkiewicz} space $\DJ(\K)$, which corresponds to the case where each $(X_i,A_i) = (\mathbb{C}P^\infty,\ast)$. Polyhedral products of the form $\ux^\K$ were first defined by Anick \cite{anick_connection} under the name of $\Gamma$-wedges, and the generalised definition above was formulated by Buchstaber and Panov \cite{BP02} (and independently in unpublished notes of Strickland). We will assume throughout the paper that $\K$ has no ghost vertices (i.e. $\{i\} \in \K$ for all $i \in [m]$).

The study of the pointed loop space of certain polyhedral products, including moment-angle complexes and Davis--Januszkiewicz spaces, has been the subject of intense study recently. In particular, integral decompositions of their pointed loop spaces in certain cases are given in \cite{PT19,St24a,St25,El26} and homological properties are determined in \cite{GPTW16,Ca24,vylegzhanin}. Moore's conjecture has been verified for moment-angle complexes in \cite{HST19}. However, this was not done in a way which also verifies Anick's conjecture, which would allow one to more explicitly enumerate the torsion-free and $p$-torsion parts of the homotopy groups for certain primes $p$.

To go further, we state Anick's conjecture in more detail. For a collection of topological spaces $\mathcal{X}$, let $\prod \mathcal{X}$ be the collection of spaces homotopy equivalent to a finite type product of spaces in $\mathcal{X}$. Let $\mathcal{P} = \{S^{2n-1},\Omega S^{m} \: : \: n \geq 1, m \geq 3, m \notin \{4,8\}\}$. Let $p$ be an odd prime. Cohen, Moore and Neisendorfer \cite{CMN79,CMN87,Nei87} defined spaces $S^{2m+1}\{p^r\}$ and $T^{2m+1}\{p^r\}$ for $r,m \geq 1$ which appear in the decomposition of the loops on a Moore space. When $p=2$ and $s \geq 2$, Cohen \cite{Co89} defined an analogous space $T^{2m+1}\{2^s\}$. Let \[\mathcal{T}_0 = \{T^{2m+1}\{p^r\},S^{2m+1}\{p^r\},T^{2m+1}\{2^s\}\::\: p \text{ odd prime, }m,r \geq 1, s \geq 2\}.\] 

Anick's conjecture asserts that if $X$ is a simply connected, finite $CW$-complex, then, localised at any sufficiently large prime, $\Omega X \in \prod(\mathcal{P} \cup \mathcal{T}_0)$. Anick's conjecture has been verified for wedges of spheres \cite{Hi55}, certain two-cones \cite{An89b}, highly connected Poincar\'e duality complexes \cite{BB18, BT22, ST25}, and certain polyhedral products associated to flag complexes \cite{PT19}, $1$-dimensional simplicial complexes \cite{St24a}, $2$-dimensional simplicial complexes \cite{St25} and certain polyhedral join products \cite{El26}. In the case of Poincar\'e duality complexes and the families of polyhedral products listed above, the loop space decompositions are integral. In the context of polyhedral products, the first localised results were given in \cite{St25}, which showed that if $p$ is a large enough prime, $\Omega \ZK \in \prod\mathcal{P}$ under certain cohomological conditions on the full subcomplexes of $\K$. 

In this paper, we show that such a decomposition holds for all moment-angle complexes when localised away from a finite number of primes. If $\K$ is a simplicial complex on $[m]$, then the set of primes is bounded above by $2^{m 2^{2^m}}$. However, this bound can be improved greatly in most cases depending on the combinatorics of $\K$. We prove the following decomposition, where the specific primes that need to be inverted will be made explicit in Theorem ~\ref{thm:anick for djk}. Moreover, we provide the enumeration of factors in Corollary \ref{crl:MACenumhtpy}.

\begin{theorem}[{Theorem \ref{thm:anick for djk}}]
\label{thm:introAnickforzk}
    If $\K$ is a simplicial complex, then localised away from an explicit finite set of primes $\Omega \ZK \in\prod\mathcal{P}$.
\end{theorem}

Theorem ~\ref{thm:introAnickforzk} can be used to show that Anick's conjecture also holds for  partial quotients of $\ZK$, i.e. quotients by a freely acting torus $H\subset (S^1)^m$ (this class includes all quasitoric manifolds \cite{dj}) and, more generally, universal covers of toric orbifolds (see Section ~\ref{sec:partquoandorbi}).

In Section ~\ref{sec:LoopCaa}, a novel technique is developed which reduces the study of the pointed loop spaces of general polyhedral products to the study of moment-angle complexes. We give four applications of this method, but we expect it to be useful in other contexts. We first show that the torsion in loop homology and action of the Steenrod algebra on loop cohomology of polyhedral products of the form $\uxa^\K$ are determined by the moment-angle complex case, and the ingredient spaces (Theorem ~\ref{thm:uxaFPT}). The action of the Steenrod algebra on the cohomology of moment-angle complexes was studied in \cite{AGIJL25}, and in some cases this induces an action on $H^*(\OZK)$. Theorem ~\ref{thm:uxaFPT} shows that this gives rise to a non-trivial action on $H^*(\Omega \uxa^\K)$ with mild hypotheses on $(X_i,A_i)$ (see Example ~\ref{ex:torsionandSteenrod}).

Next, we apply this method to give loop space decompositions of a wider family of polyhedral products. In particular, we show $\Omega \caa^{\K} \in \prod(\mathcal{P} \cup \mathcal{T}_0)$ whenever $\Omega \ZK \in \prod (\mathcal{P} \cup \mathcal{T}_0)$ and the suspension of each $A_i$ is homotopy equivalent to a finite type wedge of spheres and Moore spaces (Theorem ~\ref{thm:caainPifzkinP}). We then show that Anick's conjecture holds for any polyhedral product of the form $\caa^\K$ such that each $A_i$ is connected and finite
(Corollary ~\ref{cor:finitecaa}). 

Finally, we calculate the Poincar\'e series of $H_*(\Omega \uxa^\K;\k)$ where $\k$ is any field and $\uxa^\K$ is simply connected (Theorem ~\ref{thm:PSforuxa}). This vastly generalises a result of Cai \cite{Ca24} who considered polyhedral products of the form  $\ux^\K$, where $\K$ is a flag complex. If $\Omega \uxa^\K \in \prod\mathcal{P}$, this allows us to enumerate the spheres and loops on spheres that appear in such a decomposition.  

The structure of the paper is as follows. In Section ~\ref{sec:Plocaldecomp}, we prove sufficient conditions on a space $X$ for its pointed loop space to be in $\prod\mathcal{P}$ localised away from a set of primes $P$, and explain how to compute localised homotopy groups of $X$ in terms of homotopy groups of spheres and loop homology of $X$. In Section \ref{sec:LoopMAC}, we focus on the case of a moment-angle complex and prove Theorem ~\ref{thm:introAnickforzk}. We use this in Section ~\ref{sec:partquoandorbi} to determine the homotopy groups of partial quotients and toric orbifolds in terms of the homotopy groups of their associated moment-angle complex. Finally, in Section ~\ref{sec:LoopCaa}, we consider loop spaces of more general polyhedral products and develop the methods to reduce the study of these to the loop spaces of moment-angle complexes. 

The authors would like to thank Steven Amelotte for pointing out the papers of Anick on $R$-local homotopy theory, which were instrumental in proving the main result, and the service team of Comptes Rendus for sending a scanned copy of \cite{backelin}.

\subsection*{AI Declaration}

No generative AI tool was used in the mathematical creative process. 
Google Gemini was used for proofreading.

\section{\texorpdfstring{$P$}{P}-local decompositions of loop spaces}
\label{sec:Plocaldecomp}
Let $X$ be a simply connected space of finite type. Rational homotopy groups of $X$ are determined by the loop homology algebra $H_*(\Omega X;\QQ)$, and the based loop space $\Omega X$ is rationally homotopy equivalent a product of spheres and loops on spheres. In this section, we show that similar results hold $p$-locally, under certain assumptions on $H_*(\Omega X)$. In particular, we give a simple proof of a theorem of Hilton--Serre--Baues, following Anick.

In favourable situations, this method expresses homotopy groups of $X$ in terms of $H_*(\Omega X;\QQ)$ and homotopy groups of spheres (see Corollary \ref{crl:sphere-enumeration}).

For a space $X$, we denote by $\tau(X)$ the set of primes $p$ appearing as $p$-torsion in $H_*(X;\ZZ)$. For a set of primes $P$, denote by $\ZZ[1/P]$ the subring of $\QQ$ consisting of fractions whose denominator is a product of elements of $P$. Note that $\ZZ[1/\varnothing] = \ZZ$, and if $p$ is a prime and $P$ is the set of all primes without $p$, then $\ZZ[1/P] = \ZZ_{(p)}$.

\subsection{Anick's \texorpdfstring{$R$}{R}-local Milnor--Moore theorem}

If $X$ is a simply connected $CW$-complex, the rational homotopy groups, $\pi_{*+1}(X) \otimes_\ZZ \QQ\cong\pi_*(\Omega X)\otimes_\ZZ\QQ$, have a Lie algebra structure given by Whitehead products. A fundamental result in rational homotopy theory is the Milnor--Moore isomorphism $U(\pi_*(\Omega X) \otimes_\ZZ \QQ) \cong H_*(\Omega X;\QQ)$ of Hopf algebras, where $U$ is the universal enveloping algebra functor. In this section, we state an approximation to the Milnor--Moore theorem which applies to subrings of $\QQ$. This was proved by Anick \cite{anick_mm}, although it was not stated explicitly in this form. 

Let $R$ be a subring of $\QQ$. By ``dgl'' we mean a differential Lie algebra over $R$. A dgl $L=\bigoplus_{k\geq 0}L_k$ is \emph{$r$-reduced} if $L_{<r}=0$. Fix an integer $\rho\geq 1$ such that $1/(\rho-1)!\in R$, and an integer $r\geq 1$. To each $r$-connected space $X$, Anick assigns an $r$-reduced dgl $\mathcal{K}(X)$, see \cite[Section 2]{anick_mm}.

\begin{lemma}[{\cite[Theorem 3.7]{anick_mm}}]
\label{lem:naturalmaptodgl}
Let $d=\min(r+2\rho-3,r\rho-1)$. There is a natural map of Lie algebras
$$\pi_*(\Omega X)\otimes_{\ZZ}R\to H(\mathcal{K}(X)),$$ which is defined in degrees $<r\rho$, an isomorphism
in degrees $<d$ and an epimorphism in degree $d$.\qed
\end{lemma}

Define $N(r,\rho):=2\lceil (r+2)/2\rceil\rho-2$. Clearly, $N(r,\rho)\geq (r+2)\rho-2\geq r\rho$.
\begin{lemma}[{\cite[Lemma 1.12]{anick_mm}}]
\label{lem:dglUEA}
Let $L$ be an $r$-reduced dgl, and suppose that $L$ and $H_{<n}(UL)$ are free $R$-modules for some $n\leq N(r,\rho)$. The natural map $UH(L)\to H(UL)$ is an isomorphism in degrees $<n$ and an epimorphism in degree $n$.\qed
\end{lemma}

\begin{lemma}[{\cite[Lemma 2.1(b)]{anick_mm}}]
\label{lem:Univtoloophom}
There is an isomorphism $H(U\mathcal{K}(X))\cong H_*(\Omega X;R)$ in degrees $<r\rho$.\qed
\end{lemma}

We obtain the following result which is not stated in \cite{anick_mm} explicitly. Note that the case where $R=\QQ$, $\rho=+\infty$ is the Milnor--Moore theorem, and the case where $R=\ZZ$, $\rho=1$ is the Hurewicz theorem.

\begin{theorem}
\label{thm:R-local Milnor--Moore}
Let $r,\rho\geq 1$ be integers and let $X$ be an $r$-connected space of finite type. Let $R\subseteq\mathbb{Q}$ be a subring such that $1/(\rho-1)!\in R$ and assume $H_{<d}(\Omega X;R)$ is a free $R$-module, where $d\leq\min(r+2\rho-3,r\rho-1)$. Then the natural map induced by the Hurewicz homomorphism 
$$U(\pi_*(\Omega X)\otimes_\ZZ R)\to H_*(\Omega X;R)$$ is an isomorphism in degrees $<d$, and an epimorphism in degree $d$.
\end{theorem}
\begin{proof}
Combining Lemmas ~\ref{lem:naturalmaptodgl}, ~\ref{lem:dglUEA} and ~\ref{lem:Univtoloophom}, the maps
$$U(\pi_*(\Omega X)\otimes_\ZZ R)\to UH(\mathcal{K}(X))\to H(U\mathcal{K}(X))\to H(\Omega X;R)$$
are defined and are isomorphisms (respectively, epimorphisms) in these degrees. The composition agrees with the Hurewicz homomorphism by \cite[Lemma 3.4]{anick_mm}.
\end{proof}
We will only use the following special case of this result. 
\begin{corollary}
\label{crl:R-local MM special}
Let $X$ be a simply connected space of finite type and $N\geq 2$ be an integer. Let $P$ be a set of primes which contains $\tau(\Omega X)$ and all primes $p$ with $p<N$. Then the natural map
$$U(\pi_*(\Omega X)\otimes_\ZZ \ZZ[1/P])\to H_*(\Omega X;\ZZ[1/P])$$
is an epimorphism in degrees $k<N$.
\end{corollary}
\begin{proof}
We apply Theorem \ref{thm:R-local Milnor--Moore} for $r=1$, $R=\ZZ[1/P]$ and $\rho=N$. By assumption, $H_*(\Omega X;\ZZ)$ has no $p$-torsion for $p\notin P$, and so $H_*(\Omega X;R)$ is a free $R$-module. Also, $1/(N-1)!\in R$ since $P$ contains all primes $p<N$. Hence Theorem ~\ref{thm:R-local Milnor--Moore} applies for $d=\min(2N-2,N-1)=N-1$.
\end{proof}

\subsection{Preliminary decompositions} 

For a collection of topological spaces $\mathcal{X}$, let $\bigvee \mathcal{X}$ (resp. $\prod \mathcal{X}$) be the collection of spaces homotopy equivalent to a finite type wedge (resp. product) of spaces in $\mathcal{X}$. Let $\mathcal{W}$ be the collection of simply connected spheres, and let $\mathcal{M}_0$ be the collection of simply connected $p^r$-Moore spaces with $p^r \neq 2$. Recall from the introduction the collections $\mathcal{P} = \{S^{2n-1},\Omega S^{m} \: : \: n \geq 1, m \geq 3, m \notin \{4,8\}\}$ and \[\mathcal{T}_0 = \{T^{2m+1}\{p^r\},S^{2m+1}\{p^r\},T^{2m+1}\{2^s\}\::\: p \text{ odd prime, }m,r \geq 1, s \geq 2\}.\] We first state some properties of the collections $\prod(\mathcal{P} \cup \mathcal{T}_0)$ and $\bigvee(\mathcal{W} \cup \mathcal{M}_0)$.

\begin{lemma}
    \label{lem:relbetweenPandW}
    The following hold:
    \begin{enumerate}
        \item If $X \in \bigvee\mathcal{W}$, then $\Omega X \in \prod\mathcal{P}$;
        \item  If $X \in
        \prod\mathcal{P}$, then $\Sigma X \in \bigvee\mathcal{W}$;
        \item If $X \in \bigvee(\mathcal{W} \cup \mathcal{M}_0)$, then $\Omega X \in \prod(\mathcal{P} \cup \mathcal{T}_0)$;
        \item If $X \in \prod(\mathcal{P} \cup \mathcal{T}_0)$, then $\Sigma X \in \bigvee(\mathcal{W} \cup \mathcal{M}_0)$.
    \end{enumerate}
\end{lemma}
\begin{proof}
     Part (1) follows from the Hilton-Milnor theorem. Part (2) follows by iterating the homotopy equivalence $\Sigma(X \times Y) \simeq \Sigma X \vee \Sigma Y \vee \Sigma(X \wedge Y)$ and James splitting $\Sigma\Omega\Sigma X\simeq\bigvee_{k\geq 1}\Sigma X^{\wedge k}$ \cite{Ja55}.   
     
     Parts (3) and (4) were proved for larger collections denoted $\prod(\mathcal{P} \cup \mathcal{T})$ and $\bigvee(\mathcal{W} \cup \mathcal{M})$ in \cite[Lemma 5.2]{St25}, but the same proof follows in our case.
\end{proof}

\begin{lemma}
    \label{lem:PTclosedunderret}
    The following hold integrally or localised away from a set of primes:
    \begin{enumerate}\item If $X$ is an $H$-space in $\prod \mathcal{P}$, and $A$ is a space which retracts off $X$, then $A \in \prod \mathcal{P}$,
    \item If $X$ is an $H$-space in $\prod (\mathcal{P}\cup \mathcal{T}_0)$, and $A$ is a space which retracts off $X$, then $A \in \prod (\mathcal{P}\cup \mathcal{T}_0)$,
    \item The collection $\bigvee\mathcal{W}$ is closed under retracts,
    \item The collection $\bigvee(\mathcal{W}\cup \mathcal{M}_0)$ is closed under retracts.
    \end{enumerate}
\end{lemma}
\begin{proof}
    Part (1) was proved integrally in \cite[Theorem 3.10]{St24a}. The same proof goes through when localised away from a set of primes. Part (2) was proved integrally for the larger collection $\prod(\mathcal{P}\cup \mathcal{T})$ in \cite[Theorem 4.9]{St25}, but the same proof holds restricting to spaces in $\prod(\mathcal{P} \cup \mathcal{T}_0)$. The same proof also works when localised away from a finite set of primes. Part (3) is well known (see for example \cite[Lemma 3.1]{Am24}), and part (4) was proved for the larger collection $\bigvee(\mathcal{W}\cup \mathcal{M})$ in \cite[Theorem 3.5]{St25}, but the same proof holds restricting to spaces in $\bigvee(\mathcal{W} \cup \mathcal{M}_0)$.
\end{proof}

\begin{remark}
\label{rmk:Hspheres}
    In some cases, the hypothesis that $X$ is an $H$-space in part (1) and (2) of Lemma ~\ref{lem:PTclosedunderret} is always satisfied if $X \in \prod \mathcal{P}$. Let $P$ be a possibly empty set of primes. If $2 \in P$, then localised away from $P$, any space in $\mathcal{P}$ is an $H$-space, and so any space $X \in \prod \mathcal{P}$ is an $H$-space. However, integrally, or if $2 \notin P$, the only spheres which are $H$-spaces are $S^1$, $S^3$ and $S^7$. Hence, in this case, if $X \in \prod\mathcal{P}$ is an $H$-space, then there are no spheres of the form $S^{2n-1}$ with $n \notin \{1,2,4\}$ in the product decomposition for $X$.
\end{remark}

\subsection{\texorpdfstring{$P$}{P}-local Hilton--Serre--Baues theorem}
Now we give conditions on a simply connected space $X$ so that $\Omega X \in \prod\mathcal{P}$ localised away from a finite set of primes $P$. The following result is a slight generalisation of the Hilton--Serre--Baues theorem \cite[Ch. V, Lemma (3.10)]{baues_book}, originally proved in the case where $2,3\in P$. Our proof is a modification of Anick's argument, see \cite[Theorem 4]{anick_loop}.
 
\begin{theorem}
\label{thm:R-local BHS}
    Let $X$ be a simply connected space of finite type.
    Let $P$ be a set of primes such that $\tau(\Omega X)\subseteq P$, i.e. $H_*(\Omega X;\ZZ)$ has no $p$-torsion for any prime $p\notin P$. Suppose $H_*(\Omega X;\mathbb{Z}[1/{P}])$ is generated as a $\ZZ[1/P]$-algebra by the image of the $\ZZ[1/P]$-local Hurewicz homomorphism. Then $\Omega X \in\prod\mathcal{P}$ localised away from $P$.
\end{theorem}
\begin{proof}
Localise away from the primes in $P$.  By assumption, $H_*(\Omega X;\mathbb{Z}[1/P])$ is torsion free.

Since $X$ is of finite type, each group $\pi_n(X)$ is finitely generated. By assumption, $H_*(\Omega X;\ZZ[1/P])$ is multiplicatively generated by a set of elements $\{x_i:i\in\mathcal{I}\}$, and there are finitely many generators in each degree. Moreover, $x_i=(f_{i})_*([S^{n_i}])\in H_{n_i}(\Omega X;\ZZ[1/P])$ for some maps $f_i:S^{n_i}\to\Omega X$. Let $f=\vee_if_i:\bigvee_{i\in\mathcal{I}} S^{n_i} \rightarrow \Omega X$. By the universal property of the James construction, there is a unique extension of $f$ to an $H$-map $\overline{f}:\Omega \left(\bigvee_{i\in\mathcal{I}} S^{n_i +1}\right) \rightarrow \Omega X$. Since $H_*(\Omega X;\mathbb{Z}[1/P])$ is generated by $\{x_i\}_{i \in \mathcal{I}}$, $\overline{f}$ induces a surjection on homology with $\ZZ[1/P]$ coefficients by the Bott-Samelson theorem. 
    
Suspending, there is a homotopy equivalence $\Sigma \Omega( \bigvee_{i\in\mathcal{I}} S^{n_i +1}) \simeq W$ for some finite type wedge of spheres $W$, and up to this homotopy equivalence, the map $\Sigma \overline{f}:\Sigma \Omega( \bigvee_{i\in\mathcal{I}} S^{n_i +1}) \rightarrow \Sigma \Omega X$ becomes a map $g:W \rightarrow \Sigma \Omega X$ which induces a surjection on homology with $\ZZ[1/P]$ coefficients. Consider the commutative diagram \[\begin{tikzcd}
	{\pi_n(W) \otimes \mathbb{Z}[1/P]} & {\pi_n(\Sigma \Omega X) \otimes \mathbb{Z}[1/P]} \\
	{H_n(W) \otimes \mathbb{Z}[1/P]} & {H_n(\Sigma \Omega X) \otimes \mathbb{Z}[1/P],}
	\arrow[from=1-1, to=1-2]
	\arrow["{h_W \otimes 1}", from=1-1, to=2-1]
	\arrow["{h_{\Sigma \Omega X} \otimes 1}", from=1-2, to=2-2]
	\arrow["{g_* \otimes 1}", from=2-1, to=2-2]
\end{tikzcd}\] where $h_W$ and $h_{\Sigma \Omega X}$ are the Hurewicz homomorphisms for $W$ and $\Sigma \Omega X$ respectively. Since $W$ is a wedge of spheres, $h_W \otimes 1$ is surjective. Commutativity of the diagram with surjectivity of $g_* \otimes 1$ implies that $h_{\Sigma \Omega X} \otimes 1$ is surjective. Since $H_*(\Sigma \Omega X;\ZZ[1/P])$ is a free $\ZZ[1/P]$-module, for each generator $y$ of a $\ZZ[1/P]$ summand, there is a map $f_y:S^{n_y} \rightarrow \Sigma \Omega X$ which maps a generator of $H_{n_y}(S^{n_y})$ to $y$. Let $h =\vee_y f_y: \bigvee_{y} S^{n_y} \rightarrow \Sigma \Omega X$. By definition of $h$, $h$ induces an isomorphism on homology with coefficients in $\ZZ[1/P]$, and so Whitehead's theorem implies $h$ is a homotopy equivalence, Looping, the Hilton-Milnor theorem implies that $\Omega \Sigma \Omega X \in \prod\mathcal{P}$, but $\Omega X$ retracts off $\Omega \Sigma \Omega X$, implying by part (1) of Lemma ~\ref{lem:PTclosedunderret} that $\Omega X \in\prod\mathcal{P}$.
\end{proof}

In practice, it is easier to compute loop homology with coefficients in a field. We will use the following result which relates generation of $\FF_p$-algebras and $\ZZ[1/P]$-algebras.

\begin{lemma}
\label{lemma:Zp implies ZP for algebras}
    Let $N\geq 0$ be an integer and $A = \bigoplus_{n \geq 0} A_n$, $A_0=\ZZ\cdot 1$ be a graded $\ZZ$-algebra such that each $A_n$ is a finitely generated abelian group. Let $P$ be a set of primes such that for all $p \notin P$, the $\FF_p$-algebra $A \otimes \FF_p$ is multiplicatively generated by elements of degree less than $N$. Then the $\ZZ[1/P]$-algebra $A \otimes \ZZ[1/P]$ is multiplicatively generated by elements of degree less than $N$.
\end{lemma}
\begin{proof}
       Denote by $S\subseteq A$ the subalgebra generated by elements of degree less than $N$. For each $n\geq 0$, there is a short exact sequence of abelian groups
    $$0\to S_n\to A_n\overset{\pi}\longrightarrow Q_n\to 0,$$ where $Q_n = A_n / S_n$.

    Let $x\in A_n$ be an element and let $x_p\in A_n\otimes\FF_p$ be its reduction mod $p$, where $p \notin P$. By assumption, $x_p$ is a non-commutative polynomial on elements of $A\otimes\FF_p$ of degree less than $N$. Taking their integral lifts, we obtain that $x=x'+pr$, where $x'\in S_n$ and $r\in A_n$. It follows that $\pi(x)=p\cdot\pi(r)$ for some $r$. Since $x\in A_n$ and $p \notin P$ were arbitrary, $Q_n=p\cdot Q_n$ for all $p \notin P$. The abelian group $Q_n$ is finitely generated, so there is no $p$-torsion for any $p \notin P$ and no free part. If $P = \varnothing$, this implies $Q_n$ is trivial and so $A$ is generated by elements of degree less than $N$. If $P \neq \varnothing$, there is an integer $M\neq 0$ coprime with any $p \notin P$ such that $M\cdot Q_n=0$. It follows that $M\cdot x\in S_n$.

    Now let $y\in A\otimes\ZZ[1/P]$ be a homogeneous element. Consider the map $f:A\to A\otimes\ZZ[1/P]$. There is an integer $M'\neq 0$ coprime with any $p \notin P$ such that $M'y=f(x)$ for some $x\in A$ (i.e. $M'$ is the denominator of $y$). By the argument in the previous paragraph, there is an integer $M$ coprime with 
 any $p\notin P$ such that $Mx$ is a non-commutative polynomial on elements of $A$ of degree $<N$, hence $Mf(x)=MM'y$ is a non-commutative polynomial on elements of $A\otimes\ZZ[1/P]$ of degree less than $N$. The number $MM'$ is invertible in $\ZZ[1/P]$, so $y$ is a non-commutative polynomial on elements of $A\otimes\ZZ[1/P]$ of degree less than $N$.
\end{proof}
\begin{lemma}
\label{lemma:Zp implies ZP for loop homology}
    Let $N\geq 0$ be an integer and $X$ be a simply connected space of finite type. Let $P$ be a set of primes such that for any prime $p\notin P$,
    \begin{enumerate}
        \item $H_*(\Omega X;\ZZ)$ has no $p$-torsion;
        \item The $\FF_p$-algebra $H_*(\Omega X;\FF_p)$ is multiplicatively generated by elements of degree less than $N$.
    \end{enumerate}
    Then the $\ZZ[1/P]$-algebra $H_*(\Omega X;\ZZ[1/P])$ is multiplicatively generated by elements of degree less than $N$.
\end{lemma}
\begin{proof}
By assumption, $A=H_*(\Omega X;\ZZ)$ is a connected graded $\ZZ$-algebra of finite type. Since there is no $p$-torsion for $p\notin P$, $H_*(\Omega X;\FF_p)\cong A\otimes\FF_p$, and $H_*(\Omega X;\ZZ[1/P])\cong A\otimes\ZZ[1/P]$. Hence Lemma \ref{lemma:Zp implies ZP for algebras} applies.
\end{proof}

\subsection{Applications of Hilton--Serre--Baues theorem}
The following proposition gives a sufficient condition for $\Omega X\in\prod\mathcal{P}$ localised away from $P$. 

\begin{proposition}
\label{prp:bounded generation imply anick}
    Suppose that $X$ is a simply connected space of finite type, $Q$ is a set of primes and $N\geq 2$ is an integer such that
    \begin{itemize}
        \item For any $p\notin Q$, the $\FF_p$-algebra $H_*(\Omega X;\FF_p)$ is generated by elements of degree less than $N$.
    \end{itemize}
    Then $\Omega X\in\prod\mathcal{P}$, localised away from $P=Q\cup\tau(\Omega X)\cup\{\text{prime }p:~p<N\}$. 
\end{proposition}
\begin{proof}
Since $\tau(\Omega X)\subset P$, the algebra $H_*(\Omega X;\ZZ[1/P])$ is a free $\ZZ[1/P]$-module. By Corollary \ref{crl:R-local MM special} and $\{p:p<N\}\subset P$, the homology groups $H_{<N}(\Omega X;\ZZ[1/P])$ are in the image of the map $U(\pi_*(\Omega X)\otimes\ZZ[1/P]) \rightarrow H_{*}(\Omega X;\ZZ[1/P])$ induced by the Hurewicz homomorphism. On the other hand, by assumption and Lemma \ref{lemma:Zp implies ZP for loop homology}, the $\ZZ[1/P]$-algebra $H_*(\Omega X;\ZZ[1/P])$ is multiplicatively generated by $H_{<N}(\Omega X;\ZZ[1/P])$. It follows that this algebra is multiplicatively generated by the Hurewicz image. Thus Theorem \ref{thm:R-local BHS} implies $\Omega X \in \prod\mathcal{P}$ localised away from $P$.
\end{proof}
\begin{remark}
Proposition \ref{prp:bounded generation imply anick} is always applicable for $Q=\{\text{all primes}\}$. This reproves the classical result that $\Omega X\in\prod\mathcal{P}$ after rationalization. In our applications, $P$ will be finite.
    \end{remark}
Finally, we can describe localised homotopy groups of a simply connected space $X$ satisfying these conditions in terms of its loop homology and homotopy groups of spheres.
\begin{corollary}
\label{crl:sphere-enumeration}
Under assumptions of Proposition \ref{prp:bounded generation imply anick}, there is a homotopy equivalence
$$\Omega X\simeq (S^1)^{\times A}\times (S^3)^{\times B}\times (S^7)^{\times C}\times \prod_{n\geq 3}^{n\neq 4,8}(\Omega S^n)^{\times D_n}$$
away from $P$, where the integers $A,B,C,D_n$ are determined by the identity
$$1/F(H_*(\Omega X;\QQ);t)=\prod_{n\geq 3}^{n\neq 4,8}(1-t^{n-1})^{D_n} / (1+t)^A(1+t^3)^B(1+t^7)^C.$$
In particular,
$$\pi_N(X)\otimes \ZZ[1/P]\simeq \left(\pi_{N-1}(S^3)^{\oplus B}\oplus\pi_{N-1}(S^7)^{\oplus C}\oplus\bigoplus_{3\leq n\leq N}^{n\neq 4,8}\pi_N(S^n)^{\oplus D_n}\right) \otimes \mathbb{Z}[1/P],\quad N\geq 3.$$
\end{corollary}
\begin{proof}
A stated homotopy equivalence follows since $\Omega X\in\prod\mathcal{P}$ away from $P$, and $X$ is of finite type. To compute the numbers $A,B,C,D_n$, we calculate Poincar\'e series twice, using the formulas
\[F(H_*(\prod_i X_i);t)=\prod_i P(H_*(X_i);t),\quad F(H_*(S^k);t)=1+t^k,\quad F(H_*(\Omega S^k);t)=1/(1-t^{k-1}).\qedhere\]
\end{proof}
\begin{remark}
Recall that $\Omega S^{2k}\simeq S^{2k-1}\times \Omega S^{4k-1}$ away from $2$. Hence, under assumptions of Proposition \ref{prp:bounded generation imply anick}, we obtain a further decomposition when localised away from $P\cup\{2\}$:
$$\Omega X\simeq \prod_{k\geq 0}(S^{2k+1})^{\times E_k}\times (\Omega S^{2k+3})^{\times F_k},\quad 1/F(H_*(\Omega X;\QQ);t)=\prod_{k\geq 0}(1-t^{2k})^{F_k}/(1+t^{2k})^{E_k}.$$
\end{remark}

\section{Loop spaces of moment-angle complexes}
\label{sec:LoopMAC}
In this section, $\K$ is a simplicial complex on $[m]$
without ghost vertices, and $\k$ is a field.
\subsection{Algebraic preliminaries} 
The $\Zm$-graded $\k$-algebra
$$\k[\K]:=\k[v_1,\dots,v_m]/(\prod_{i\in J}v_i=0,~J\notin\K),\quad\deg v_i:=2e_i\in\Zm$$
is called the \emph{Stanley--Reisner ring} associated with $\K$. For a multiindex $\alpha=\sum_{i=1}^m\alpha_ie_i\in\Zm$, we denote $|\alpha|:=\sum_i\alpha_i$. A subset $J\subseteq[m]$ is identified with the multiindex $\sum_{j\in J}e_j\in\Zm$.

A subset $I\subseteq[m]$ is called a \emph{(minimal) missing face} if $I\notin\K$, but $I\setminus\{i\}\in\K$ for any $i\in I$. We denote by $\MF(\K)$ the set of missing faces of $\K$. The ring $\k[\K]$ is a quotient of the polynomial ring by the ideal generated by monomials $v^I$, $I\in\MF(\K)$. Since $\K$ has no ghost vertices, the monomials $v^I$ are not linear.

If $A$ is a connected graded $\k$-algebra of finite type, we can compute the $\Tor$-functor $\Tor^A(\k,\k)$ as the homology of the bar construction $\overline{\mathrm{B}}(A)$ and the $\Ext$-algebra $\Ext_A(\k,\k)$ as the homology of the cobar construction $\overline{\mathrm{C}}(A)$, see \cite[\S 1]{priddy}. These complexes are dual, so for each $n\geq0 $ we obtain a duality $(\Tor^A_n(\k,\k))^*\cong\Ext^n_A(\k,\k)$ of graded $\k$-modules.

For a subset $J\subseteq[m]$, we consider the full subcomplex $\K_J:=\{I\subseteq J:I\in\K\}$ of $\K$, a simplicial complex on the vertex set $J$. Additively, $\k[\K_J]\subseteq\k[\K]$ is the direct sum of all graded components $\k[\K]_{2\alpha}$, $\alpha\in\Zm$, such that  $\alpha_i=0$ for $i\notin J$.
\begin{lemma}
\label{lmm:tor kK passing to subcomplex}
    Let $\alpha\in\Zm$ and suppose that $\alpha_i=0$ for $i\notin J$. Then $\Tor^{\k[\K]}_{n,2\alpha}(\k,\k)\cong\Tor^{\k[\K_J]}_{n,2\alpha}(\k,\k)$ for all $n\geq 0$.
\end{lemma}
\begin{proof}
    We compute the $\Tor$ functor as the homology of the bar construction $\overline{\mathrm{B}}(\k[\K])$, which has the following direct sum decomposition:
    the graded component $\Tor_{*,2\alpha}$ is computed as homology of the chain complex spanned by elements of the form $[v^{\alpha_{(1)}}\mid \dots\mid v^{\alpha_{(k)}}] $ such that
 $$k\geq 0,\quad \alpha_{(t)}\in\Zm\setminus\{0\},\quad v^{\alpha_{(t)}}\neq 0\in\k[\K],\quad \sum_{t=1}^k\alpha_{(t)}=\alpha.$$
    Since $\alpha_i=0$ for $i\notin J$, $(\alpha_{(t)})_i=0$ for $i\notin J$.
    Therefore, the chain complex is the same for $\K$ and $\K_J$, and the claim follows.
\end{proof}

\subsection{Backelin--Berglund polynomials}
\label{subsec:bb polynomials}

Following Backelin \cite{backelin} and Berglund \cite{berglund}, we consider the Poincar\'e series for the $\Tor$ functor,
$$\mathrm{P}_{\k[\K]}(x_1,\dots,x_m,z):=\sum_{n\geq 0,\alpha\in\Zm}\dim_\k \Tor^{\k[\K]}_{n,2\alpha}(\k,\k)\cdot x^\alpha z^n\in\ZZ[[x_1,\dots,x_m,z]].$$
Note that we use the convention $\deg v_i=2e_i\in\Zm$.
\begin{definition}
Consider the formal power series
\begin{equation}
\label{eqn:definition of b}
\mathrm{b}_{\k[\K]}(x_1,\dots,x_m,z):=\prod_{i=1}^m(1+x_iz) / \mathrm{P}_{\k[\K]}(x_1,\dots,x_m,z)
\end{equation}
and the expansion
$\mathrm{b}_{\k[\K]}(x_1,\dots,x_m,z)=\sum_{\alpha\in\Zm} \mathrm{b}_{\k[\K],\alpha}(z)\cdot x^\alpha,$ $\mathrm{b}_{\k[\K],\alpha}(z)\in\ZZ[[z]].$
We define $$bb_{\K,\k}(z):=\mathrm{b}_{\k[\K],[m]}(z)\in\ZZ[[z]],\quad [m]=\sum_{i=1}^m e_i,$$ and call $bb_{\K,\k}(z)$ the \emph{Backelin--Berglund polynomial} corresponding to $\K$ and $\k$. For the simplicial complex $\{\varnothing\}$ on zero vertices, we formally set $bb_{\{\varnothing\},\k}(z):=1$.
\end{definition}
\begin{proposition}
\label{prp:bb are polynomials}
Let $\K$ be a simplicial complex on the vertex set $[m]$, and let $\k$ be a field.
\begin{enumerate}
    \item $bb_{\K,\k}(z)\in\ZZ[[z]]$ is a polynomial of degree $\leq m$.
    \item
    $$\mathrm{b}_{\k[\K]}(x_1,\dots,x_m,z)=\sum_{J\subseteq[m]}bb_{\K_J,\k}(z)x^J.$$
\end{enumerate}
\end{proposition}
\begin{proof}
By \cite{backelin} (or see \cite[Theorem 5']{backelin_roos}), the formal power series $\mathrm{b}_{\k[\K]}(x_1,\dots,x_m,z)$ is a polynomial on $x_i$ and $z$ which has $x_i$-degree at most $1$ and $z$-degree at most $m$. Hence,
$$\mathrm{b}_{\k[\K]}(x_1,\dots,x_m,z)=\sum_{J\subseteq[m]}\mathrm{b}_{\k[\K],J}(z)x^J,\quad\mathrm{b}_{\k[\K],J}(z)\in\ZZ[z],~\deg\mathrm{b}_{\k[\K],J}\leq m.$$
Since $bb_{\K,\k}:=\mathrm{b}_{\k[\K],[m]}$, this proves (1). 

To prove (2), it remains to show that $\mathrm{b}_{\k[\K],J}(z)=\mathrm{b}_{\k[\K_J],J}(z)$. We first consider the expansions
$$\mathrm{P}_{\k[\K]}(x,z)=1+\sum_{\alpha\in\Zm\setminus\{0\}} g_{\K,\k,\alpha}(z)x^\alpha,\quad 1/\mathrm{P}_{\k[\K]}(x,z)=1+\sum_{\alpha\in\Zm\setminus\{0\}}h_{\K,\k,\alpha}(z)x^\alpha.$$
By the inversion formula,
$$h_{\K,\k,\alpha}=\sum_{n\geq 0}(-1)^n\sum_{\alpha=\alpha_1+\dots+\alpha_n}g_{\K,\k,\alpha_1}\cdot\dotso\cdot g_{\K,\k,\alpha_n},$$
hence
$$h_{\K,\k,J}=\sum_{n\geq 0}(-1)^n\sum_{J=J_1\sqcup\dots\sqcup J_n}g_{\K,\k,J_1}\cdot\dotso\cdot g_{\K,\k,J_n}$$
for any $J\subseteq [m]$.
By Lemma \ref{lmm:tor kK passing to subcomplex}, $g_{\K,\k,J_i}=g_{\K_{J_i},\k,J_i}=g_{\K_J,\k,J_i}$ whenever $J_i\subseteq J\subseteq[m]$. It follows that
\begin{equation}
\label{eqn:h_KkJ equals h_KJkJ}
h_{\K,\k,J}=h_{\K_J,\k,J}\text{ for any }J\subseteq [m].
\end{equation}
Multiplying the expansion of $1/\mathrm{P}_{\k[\K]}$ by $\prod_{i=1}^m (1+x_iz)=\sum_{L\subseteq[m]}z^{|L|}x^L$, we obtain:
$$\sum_\alpha\mathrm{b}_{\k[\K],\alpha}(z)x^\alpha = \sum_{\alpha=L+\beta} z^{|L|}h_{\K,\k,\beta}(z)\cdot x^\alpha.\text{ Hence }\mathrm{b}_{\k[\K],J}(z)=\sum_{J=L\sqcup M}z^{|L|}h_{\K,\k,M}(z).$$
Since $h_{\K,\k,M}=h_{\K_M,\k,M}=h_{\K_J,\k,M}$ by \eqref{eqn:h_KkJ equals h_KJkJ}, $\mathrm{b}_{\k[\K],J}(z)=\mathrm{b}_{\k[\K_J],J}(z)=:bb_{\K_J,\k}(z)$, as required.
\end{proof}
\begin{definition}
\label{def:bb reflected}
Define the \emph{reflected Backelin--Berglund polynomials} by $$\widehat{bb}_{\K,\k}(t):=t^mbb_{\K,\k}(t^{-1})\in\ZZ[t].$$
Since $bb_{\K,\k}(z)$ is a polynomial of degree $\leq m$, the same holds for $\widehat{bb}_{\K,\k}(t)$. As an example, if $bb_{\K,\k}(z)=5z^m+z-2$, then $\widehat{bb}_{\K,\k}(t)=5+t^{m-1}-2t^m$.
\end{definition}
\subsection{Berglund's formula}
In \cite{berglund}, Berglund determined Poincar\'e series of Ext-algebras of arbitrary monomial rings; in particular, this allows us to compute the polynomials $bb_{\K,\k}$ in terms of Betti numbers of explicit simplicial complexes. Berglund's constructions involve the minimal set of monomials $M$ that generate the monomial ideal, and its subsets $S$. In the case of Stanley--Reisner rings, elements of $M$ are precisely monomials of the form $v^I$, where $I\in\MF(\K)$ are minimal missing faces of $\K$. Here we give an exposition of Berglund's results using the language of simplicial complexes. For example, a squarefree monomial $v^I$ corresponds to a subset $I\subseteq[m]$, hence the least common multiple $\mathbf{m}_S$ of a set of monomials $\{v^I:I\in S\}$ is the monomial $v^{\cup S}$, where $\cup S:=\bigcup_{I\in S}I\subseteq[m]$.

Fix a simplicial complex $\K$ on the vertex set $[m]$, and a field $\k$. Let $S$ be a collection of subsets of $[m]$, i.e. $S\subseteq 2^{[m]}$. Define a simple graph $G_S$ which has vertex set given by $S$, and $I,J \subseteq [m]$ are connected by an edge if and only if $I \cap J \neq \varnothing$. We say that $S$ is \emph{connected} if $G_S$ is a connected graph. Connected components of the graph $G_S$ are of the form $G_S=G_{S_1}\sqcup\dots\sqcup G_{S_r}$, $S=S_1\sqcup \dots\sqcup S_r$, and we say that $S_1,\dots,S_r$ are the connected components of $S$. Denote by $c(S)$ the number of connected components of $S$. For example, the collection $S=\{\{1,2\},\{2,3\},\{4,5,6\}\}$ has two connected components $\{\{1,2\},\{2,3\}\}$ and $\{\{4,5,6\}\}$, so $c(S)=2$.

A subset $S\subseteq\MF(\K)$ is \emph{saturated} if, for any connected subset $T\subseteq S$, the following condition holds: if $I\in \MF(\K)$ and $I\subseteq \cup T$, then $I\in S$. Denote by $K(\K)$ the set of non-empty saturated subsets of $\MF(\K)$. Let $S\subseteq\MF(\K)$ be a subset, and let $S=S_1\sqcup\dots\sqcup S_r$ be its decomposition into connected components. Consider the following simplicial complex on the vertex set $S$:
$$\Delta'_S:=\{R\subseteq S:~\cup R\neq\cup S,~\text{ or }S_i\cap R\text{ is disconnected for some }i\in\{1,\dots,r\}\}.$$ Berglund's theorem can be stated as follows.

\begin{theorem}[{\cite[Theorem 1]{berglund}}]
\label{thm:berglund_formula}
For a field $\k$, 
$$\mathrm{b}_{\k[\K]}(x_1,\dots,x_m,z)=1+\sum_{S\in K(\K)}x^{\cup S}\cdot (-z)^{c(S)+2}\cdot F(\H_*(\Delta'_S;\k);z)\in\ZZ[[x_1,\dots,x_m,z]].$$
In particular, for any non-empty $J\subseteq [m]$,
$$bb_{\K_J,\k}(z)=\sum_{\begin{smallmatrix}
S\in K(\K):\\
\cup S=J
\end{smallmatrix}}(-z)^{c(S)+2}\cdot F(\H_*(\Delta'_S;\k);z)\in\ZZ[z].\eqno\qed$$
\end{theorem}
\begin{remark}
The sets $\{S\in K(\K):\cup S=J\}$ and $\{S\in K(\K_J):\cup S=J\}$ coincide. Since $bb_{\K_J,\k}(z):=\mathrm{b}_{\k[\K_J],J}(z)$, this gives another proof of the formula $\mathrm{b}_{\k[\K],J}(z)=\mathrm{b}_{\k[\K_J],J}(z)$.
\end{remark}

\subsection{Loop homology and reflected BB polynomials}
Homological algebra of Stanley--Reisner rings is related to loop homology of polyhedral products via the following theorem. Recall that $\DJ(\K):=(\CC P^\infty,\ast)^\K$ is the Davis--Januszkiewicz space and $\ZK:=(D^2,S^1)^\K$ is the moment-angle complex corresponding to $\K$.
\begin{theorem}[{originally \cite[(8.4)]{panov_ray}, see \cite[Theorem 1.1]{vylegzhanin22} or \cite[Proposition 6.5]{franz}}]
\label{thm:loophomologydjk}
Let $\K$ be a simplicial complex on $[m]$ and $\k$ be a principal ideal domain. There is an isomorphism of graded $\k$-algebras
$$H_*(\Omega\DJ(\K);\k)\cong\Ext_{\k[\K]}(\k,\k),\quad H_n(\Omega\DJ(\K);\k)\cong\bigoplus_{n=-i+2|\alpha|}\Ext^i_{\k[\K]}(\k,\k)_{2\alpha}.\eqno\qed$$
\end{theorem}
Theorem ~\ref{thm:loophomologydjk} implies that the loop homology of $\DJ(\K)$ admits a $\ZZ\times\Zm$-grading, given by $H_{-i,2\alpha}(\Omega\DJ(\K);\k):=\Ext^i_{\k[\K]}(\k,\k)_{2\alpha}$. We consider the multigraded Poincar\'e series
$$F(H_*(\Omega\DJ(\K);\k);t,\lambda):=\sum_{n\geq 0}\sum_{\alpha\in\Zm}\dim_\k H_{-n,2\alpha}(\Omega\DJ(\K);\k)t^{-n}\lambda^{2\alpha}\in\ZZ[[t^{-1},\lambda_1^2,\dots,\lambda_m^2]];$$
the usual Poincar\'e series $F(H_*(\Omega\DJ(\K);\k);t)$ is recovered by the substitution $t=\lambda_1=\dots=\lambda_m$. 

The algebra $H_*(\OZK;\k)$ admits a $\ZZ\times\Zm$-multigrading induced from the ambient algebra $H_*(\Omega\DJ(\K);\k)$, see \cite[Proposition 3.7]{vylegzhanin}. The multigraded Poincar\'e series of $H_*(\Omega \ZK;\k)$, $F(H_*(\OZK;\k);t,\lambda)\in\ZZ[[t^{-1},\lambda_1^2,\dots,\lambda_m^2]]$, is defined similarly and satisfies the identity
\begin{equation}
\label{eq:poincare series for odj and ozk}
F(H_*(\Omega\DJ(\K);\k);t,\lambda)=F(H_*(\OZK;\k);t,\lambda)\cdot \prod_{i=1}^m(1+t^{-1}\lambda_i^2),
\end{equation}
which is obtained using the split fibration $\OZK\to\Omega\DJ(\K)\to(S^1)^m$ from Theorem \ref{thm:uxasplitfib}.
\begin{proposition}
\label{prp:bb polynomials and loop homology of zk}
For a field $\k$ and a simplicial complex $\K$ on $[m]$, 
$$F(H_*(\Omega\DJ(\K);\k);t,\lambda)=\mathrm{P}_{\k[\K]}(\lambda_1^2,\dots,\lambda_m^2,t^{-1})\in\ZZ[[t^{-1},\lambda_1^2,\dots,\lambda_m^2]],$$
$$1/F(H_*(\OZK;\k);t,\lambda)=\sum_{J\subseteq[m]}bb_{\K_J,\k}(t^{-1})\lambda^{2J}=\sum_{J\subseteq[m]}\widehat{bb}_{\K_J,\k}(t)t^{-|J|}\lambda^{2J}.$$
In particular, the ordinary Poincar\'e series satisfy $F(H_*(\Omega\DJ(\K);\k);t)=\mathrm{P}_{\k[\K]}(t^2,\dots,t^2,t^{-1})$ and
$$1/F(H_*(\OZK;\k);t)=\sum_{J\subseteq[m]}\widehat{bb}_{\K_J,\k}(t)t^{|J|}\in\ZZ[t].$$
\end{proposition}
\begin{proof}
The formulas for $\DJ(\K)$ follow from the isomorphism $H_{-n,2\alpha}(\Omega\DJ(\K);\k)\cong\Ext^n_{\k[\K]}(\k,\k)_{2\alpha}$ of Theorem~\ref{thm:loophomologydjk} and from the duality between $\Ext_{\k[\K]}(\k,\k)$ and $\Tor^{\k[\K]}(\k,\k)$. The formula for $1/F(H_*(\OZK;\k);t)$ follows from \eqref{eqn:definition of b}, \eqref{eq:poincare series for odj and ozk} and part (2) of Proposition \ref{prp:bb are polynomials}.
\end{proof}
\begin{proposition}
\label{prp:bb polynomials for flag complexes}
If $\K$ is a flag complex, then $\widehat{bb}_{\K,\k}(t)=1-\chi(\K)$, and $\widehat{bb}_{\K_J,\k}(t)=1-\chi(\K_J)$.
\end{proposition}
\begin{proof}
By \cite[Theorem 4.8]{vylegzhanin22}
$$1/F(H_*(\OZK;\k);t,\lambda)=\sum_{J\subseteq[m]}(1-\chi(\K_J))t^{-|J|}\lambda^{2J}$$
whenever $\K$ is a flag complex, so the statement follows from Proposition ~\ref{prp:bb polynomials and loop homology of zk}. (A simpler proof can be given using the multigraded version of the \emph{Fr\"oberg formula} which relates Poincar\'e series of $H_*(\Omega\DJ(\K);\k)$ and $\k[\K]$ in the flag case, see \cite[Proposition 8.5.4]{BP15}.)
\end{proof}

\subsection{FPT property of moment-angle complexes}
We first give an upper bound on the size of torsion in homology of finite simplicial complexes. If $A$ is an abelian group, denote by $\Tors A$ its torsion subgroup.
\begin{theorem}[{\cite[Lemma 3]{hkp}}]
\label{thm:hkp_torsion_bound}
Let $\K$ be a $d$-dimensional simplicial complex on $k$ vertices, $d\geq 2.$ Then $|\Tors H_{d-1}(\K;\ZZ)|\leq \sqrt{d+1}^{\binom{k}{d}}.$ \qed
\end{theorem}
\begin{corollary}
\label{c:simp_torsion_bound}
Let $\K$ be a simplicial complex on $k'\leq k$ vertices. Then $\max\tau(\K)\leq \sqrt{k+1}^{\binom{k}{\lfloor k/2\rfloor}}.$ 
\end{corollary}
\begin{proof}
    Without loss of generality, $k'=k$. Let $p\in\tau(\K)$. Then $H_{d-1}(\K;\ZZ)$ contains $p$-torsion for some $d\in\{2,\dots,k\}.$ Note that $H_{d-1}(\K;\ZZ)\cong H_{d-1}(\sk_d\K;\ZZ),$ where $\sk_d\K$ is the $d$-skeleton; hence $$p\leq |\Tors H_{d-1}(\sk_d\K;\ZZ)|\leq\sqrt{d+1}^{\binom{k}{d}}$$ by Theorem \ref{thm:hkp_torsion_bound}. Finally, there are inequalities $\sqrt{d+1}\leq \sqrt{k+1},$ $\binom{k}{d}\leq\binom{k}{\lfloor k/2\rfloor}$, and the claimed upper bound follows.
\end{proof}
Corollary \ref{c:simp_torsion_bound} can be used to bound the $p$-torsion appearing in $H_*(\ZK)$.
\begin{proposition}
Let $\K$ be a simplicial complex on $m$ vertices. Then
$$\tau(\ZK)=\bigcup_{J\subseteq[m]}\tau(\K_J)\subseteq\{p:p\leq \sqrt{m+1}^{\binom{m}{\lfloor m/2\rfloor}}\}\subseteq\{p:p\leq m^{2^m}\}.$$
\end{proposition}
\begin{proof}
Observe that $\binom{m}{i}<2^m$ and $\sqrt{k+1}\leq k$. The result then follows from Corollary \ref{c:simp_torsion_bound} and the isomorphism $H^*(\ZK;\ZZ)\cong\bigoplus_{J\subseteq[m]}\H^{*-|J|-1}(\K_J;\ZZ)$ \cite[Theorem 4.5.8]{BP15}.
\end{proof}

Since the Backelin-Berglund polynomials can be written in terms of the homology of a simplicial complex, we can also use Corollary \ref{c:simp_torsion_bound} to show a bound on the $p$-torsion which can appear in $H_*(\Omega \ZK)$. We first show that the rational Backelin-Berglund polynomials are equal to the mod-$p$ Backelin-Berglund polynomials if $p$ is a sufficiently large prime.

\begin{proposition}
\label{prp:bb polymonials stabilise}
Let $\K$ be a simplicial complex on $m$ vertices. Then $bb_{\K,\QQ}(z)=bb_{\K,\FF_p}(z)$ for any prime $p>f(m)$, where
$$f(m):=\sqrt{k+1}^{\binom{k}{\lfloor k/2\rfloor}},\quad k=\binom{m}{\lfloor m/2\rfloor}.$$ Moreover, $f(m)<2^{m2^{2^m}}$.

\end{proposition}
\begin{proof}
By Berglund's formula (Theorem \ref{thm:berglund_formula}), $bb_{\K,\k}(z)$ is determined by Poincar\'e series of auxiliary simplicial complexes $\Delta'_S$ on vertex sets $S$, $S\subseteq\MF(\K)$. Since $|\MF(\K)|\leq\binom{m}{\lfloor m/2\rfloor}=k$ by Sperner's theorem, $\max\tau(\Delta'_S)\leq f(m)$ by Corollary \ref{c:simp_torsion_bound}. Hence $H_*(\Delta'_S;\ZZ)$ has no $p$-torsion, so $F(H_*(\Delta'_S;\FF_p);t)=F(H_*(\Delta'_S;\QQ);t)$ by the universal coefficient theorem. 

For the second part, $f(m)<2^{m2^{2^m}}$ follows from $\binom{m}{i}<2^m$ and $\sqrt{k+1}\leq k$.
\end{proof}
\begin{theorem}
\label{thm:djk has fpt}
    Let $\K$ be a simplicial complex on $m$ vertices. Then $$\tau(\Omega\DJ(\K))=\tau(\OZK)\subseteq\{p:p\leq f(m)\}\subseteq\{p:p\leq 2^{m2^{2^m}}\}.$$ In particular, $\tau(\OZK)$ is finite, i.e. $\OZK$ and $\Omega\DJ(\K)$ have FPT.
\end{theorem}
\begin{proof}
Since $\Omega\DJ(\K)\simeq (S^1)^{\times m}\times\OZK$ by Theorem \ref{thm:uxasplitfib}, $\tau(\OZK)=\tau(\Omega\DJ(\K))$. Let $\k$ be a field and $p>f(m)$ be a prime number. By Proposition \ref{prp:bb polynomials and loop homology of zk}, the Poincar\'e series $F(H_*(\OZK;\k);t)$ is expressed in terms of Backelin--Berglund polynomials $bb_{\K_J,\k}(z)$, $J\subseteq[m]$. Proposition \ref{prp:bb polymonials stabilise} implies $bb_{\K_J,\QQ}(z)=bb_{\K_J,\FF_p}(z)$, so $F(H_*(\OZK;\QQ);t)=F(H_*(\OZK;\FF_p);t)$ for $p>f(m)$. By the universal coefficient theorem, $H_*(\OZK;\ZZ)$ has no $p$-torsion.
\end{proof}

It was shown in \cite[After Theorem 6.4]{St25} that $H_*(\Omega \ZK)$ can contain arbitrarily large torsion in homology. Let $t(m)$ be the smallest integer such that, for any simplicial complex $\K$ on $\leq m$ vertices, there is no $p$-torsion in $H_*(\OZK;\ZZ)$ for $p>t(m)$. Equivalently, $t(m)=\max\bigcup_\K \tau(\OZK)$ where the union is over all  simplicial complexes $\K$ on $k\leq m$ vertices. Theorem \ref{thm:djk has fpt} implies $\log_2 t(m)\leq \log_2 f(m)\leq m\cdot 2^{2^m}$. We provide a lower bound, showing that $\log_2 t(m)$ grows faster than any polynomial.

\begin{theorem}[see {\cite[Theorem 1]{newman}}]
\label{thm:New18}
Let $d\geq 2$ be an integer.
Then there exists a constant $C_d>0$ with the following property: for any finite abelian group $G$, there is a $d$-dimensional simplicial complex $X$ on $k\leq C_d\cdot\log(|G|)^{1/d}$ vertices such that the torsion subgroup of $H_{d-1}(X;\ZZ)$ is isomorphic to $G$.\qed
\end{theorem}

A simplicial complex $\K$ is called $k$-neighbourly, where $k\geq 0$ is an integer, if each set of $k+1$ vertices spans a simplex.

\begin{proposition}
Let $d\geq 2$ be an integer. There is a constant $\varepsilon_d>0$ with the following property: if $m\geq 1$ and $p\leq 2^{\varepsilon_d m^d}$ is a prime number, then there is a $(d-1)$-neighbourly simplicial complex $\K$ on $m'\leq m$ vertices such that the groups $H_{d-1}(\K;\ZZ)$, $H_{d+m'}(\ZK;\ZZ)$ and $H_{k(d+m'-1)}(\OZK;\ZZ)$ for $k\geq 1$ contain $p$-torsion. Therefore, $\log_2 t(m)\geq \varepsilon_d m^d$.
\end{proposition}
\begin{proof}
    We apply Theorem ~\ref{thm:New18} for $G=\FF_p$. Let $\varepsilon_d:=C_d^{-d}$. We obtain a $d$-dimensional simplicial complex $X$ on $$m'\leq C_d\cdot\log(p)^{1/d}\leq C_d\cdot \varepsilon_d^{1/d}m=m$$ vertices such that the torsion subgroup of $H_{d-1}(X;\ZZ)$ is isomorphic to $\ZZ/p$. Let $\K:=X\cup\sk_{d-1}\Delta$ be $X$ with all $\leq(d-1)$-dimensional simplices added. The quotient $|\K|/|X|$ is a $(d-1)$-dimensional $CW$-complex, so $H_d(\K,X)=0$. It follows that $H_{d-1}(X)\to H_{d-1}(\K)$ is injective, so $\ZZ/p$ is a subgroup of $H_{d-1}(\K;\ZZ)$.

    Since $\K$ is a $(d-1)$-neighbourly $d$-dimensional complex, by \cite[Theorem 10.9 + Theorem 1.3]{fat_wedge}, there is a homotopy equivalence $\ZK\simeq\bigvee_{J\subseteq[m']}\Sigma^{|J|+1}|\K_J|$. Therefore, $\Sigma^{1+m'}|\K|$ retracts off $\ZK$, and so $\Omega\Sigma^{1+m'}|\K|$ retracts off $\OZK$. Hence, by the James splitting \cite{Ja55}, $H_{k(d+m'-1)}(\OZK;\ZZ)$ has $p$-torsion.
\end{proof}
In all known cases, there is an inclusion $\tau(\Omega\ZK)\subseteq \tau(\ZK)$. It would be interesting to know whether this holds in general.
\begin{problem}
Does there exist a simplicial complex $\K$ such that $\tau(\Omega\ZK)\not\subseteq \tau(\ZK)$?
\end{problem}

\subsection{\texorpdfstring{$P$}{P}-local decompositions of looped Davis--Januszkiewicz spaces}
We now show that Proposition \ref{prp:bounded generation imply anick} applies to $\DJ(\K)$, using the following result of Backelin and Roos on ``stabilisation'' for monomial rings. We state the result only in the case of Stanley--Reisner rings.
\begin{theorem}[{\cite[Theorem 5']{backelin_roos}}]
\label{thm:backelin_roos}
Let $\k$ be a field and $\K$ be a simplicial complex on $[m]$. Consider the double $\Ext$-algebra $A:=\Ext_{\Ext_{\k[\K]}(\k,\k)}(\k,\k)$
(with the $\ZZ_{\geq 0}\times\ZZ_{\geq 0}\times 2\ZZ_{\geq 0}^m$-grading), and let $V_i\in A^{1,1,2e_i}$ be the elements which naturally correspond to the generators $v_i\in\k[\K]$, $i=1,\dots,m$. Then the multiplication by $V_i$ gives isomorphisms
$$A^{s,t,2\beta}\to A^{s+1,t+1,2(\beta+e_i)}$$
whenever $s,t\in\ZZ,$ $\beta\in\ZZ^m$ and $\beta_i\geq 1$.\qed
\end{theorem}
\begin{corollary}
\label{cor:toralgindexing}
Let $\k$ be a field, $n\geq 0$, $\alpha\in\Zm$, and suppose that $\Tor^{H_*(\Omega\DJ(\K);\k)}_{1,-n,2\alpha}(\k,\k)\neq 0$. Then $\alpha_i\leq 1$ for all $i=1,\dots,m$.
\end{corollary}
\begin{proof}
Since $\Tor^R_i(\k,\k)$ is dual to $\Ext_R^i(\k,\k)$ and $H_{-n,2\alpha}(\Omega\DJ(\K);\k)\cong\Ext_{\k[\K]}^{n,2\alpha}(\k,\k)$ by Theorem \ref{thm:loophomologydjk}, the $\k$-modules $\Tor^{H_*(\Omega\DJ(\K);\k)}_{i,-n,2\alpha}(\k,\k)$ and $A^{i,n,2\alpha}$ are dual. We need to show that $A^{1,n,2\alpha}\neq0$ implies $\alpha_i\leq 1$. Suppose $\alpha_i\geq 2$. By Theorem \ref{thm:backelin_roos}, $A^{0,n-1,2(\alpha-e_i)}\cong A^{1,n,2\alpha}\neq 0$. However, $\Ext_{\k[\K]}(\k,\k)$ is a connected algebra, so $\bigoplus_{n',\alpha'}A^{0,n',2\alpha'}=A^{0,0,0}=\k$. This implies $\alpha=e_i$, so $\alpha_i=1$, a contradiction.
\end{proof}
\begin{proposition}
\label{prp:hodj bound on degrees}
Let $\k$ be a field. The algebra $H_*(\Omega \DJ(\K);\k)$ is multiplicatively generated by elements of degree less than $2m$.
\end{proposition}
\begin{proof}
For a connected $\k$-algebra $A$, any set of additive generators of the $\k$-module $\Tor^A_1(\k,\k)$ gives a set of multiplicative generators for $A$ (see e.g. \cite[Theorem A.6(1)]{vylegzhanin}). By Corollary ~\ref{cor:toralgindexing}, $H_*(\Omega\DJ(\K);\k)$ is generated by elements of bidegree $(-n,2\alpha)$ such that $n\geq 0$ and $\alpha_1,\dots,\alpha_m\leq 1$. Also, $H_{0,2\alpha}(\Omega\DJ(\K);\k)\cong\Hom_{\k[\K]}(\k,\k)^{2\alpha}=0$ if $\alpha\neq 0$, so $n\geq 1$. Hence the total degree of any generator is of the form $-n+2|\alpha|\leq -1+2m$. 
\end{proof}

\begin{theorem}
\label{thm:anick for djk}
Localised away from the set of primes $P=\tau(\OZK)\cup\{p:p< 2m\}$, $\Omega \DJ(\K) \in\prod\mathcal{P}$ and $\Omega \ZK \in \prod\mathcal{P}$. In particular, Anick's conjecture holds for $\DJ(\K)$ and $\ZK$, and this holds after localising away from all primes $p \leq 2^{m \cdot 2^{2^m}}$.
\end{theorem}
\begin{proof}
Localise away from $P$. The algebras $H_*(\Omega\DJ(\K);\FF_p)$ are generated in degree $< 2m$ by Proposition \ref{prp:hodj bound on degrees}, and $H_*(\Omega\DJ(\K);\ZZ)$ has no $p$-torsion for $p> 2^{m \cdot 2^{2^m}}$ by Theorem \ref{thm:djk has fpt}. Hence $\Omega \DJ(\K) \in \prod\mathcal{P}$ by Proposition ~\ref{prp:bounded generation imply anick} applied for $X=\DJ(\K)$, $Q=\varnothing$.
By Theorem ~\ref{thm:uxasplitfib}, $\Omega \ZK$ retracts off $\Omega \DJ(\K)$, and so part (1) of Lemma ~\ref{lem:PTclosedunderret} implies $\Omega \ZK \in \prod\mathcal{P}$.
\end{proof}
Following the proof of \cite[Theorem 1.2]{vylegzhanin}, for a sufficiently large prime $p$, we determine the $p$-primary component of $\pi_*(\ZK)$ in terms of homotopy groups of spheres. Recall the reflected Backelin-Berglund polynomial $\widehat{bb}_{\K_J,\QQ}(t)=t^{|J|}bb_{\K_J,\QQ}(t^{-1})$, and recall that the polynomials $bb_{\K_J,\QQ}(z)$ can be computed using Berglund's formula (Theorem \ref{thm:berglund_formula}).
\begin{corollary}
\label{crl:MACenumhtpy}
     Let $\K$ be a simplicial complex on vertex set $[m]$. Let $P$ be a set of primes containing $\tau(\OZK)$ and all primes less than $2m$. Localised away from $P$,
    there is a homotopy equivalence
    $$\Omega \ZK\simeq \prod_{n\geq 3}(\Omega S^n)^{\times D_n},$$
    where the numbers $D_n\geq 0$ are determined by the identity
    $$\prod_{n\geq 3}(1-t^{n-1})^{D_n}=\sum_{J\subseteq[m]}\widehat{bb}_{\K_J,\QQ}(t)\cdot t^{|J|}\in\ZZ[[t]].$$
    In particular, $$\pi_N(\ZK) \otimes \ZZ[1/P]\cong\bigoplus_{n=3}^N\pi_N(S^n)^{\oplus D_n} \otimes \ZZ[1/P],~N\geq 3.$$
\end{corollary}
\begin{proof}
Localise away from $P$. By Theorem \ref{thm:anick for djk}, $\OZK\in\prod\mathcal{P}$, and \cite[Lemma 6.1]{vylegzhanin} implies that $\OZK$ is homotopy equivalent to a product of loops on spheres. Since $\ZK$ is $2$-connected \cite[Proposition
4.3.5 (a)]{BP15}, there are no $\Omega S^2$ factors. The identity between formal power series is obtained by computing $1/F(H_*(\OZK;\QQ);t)$ twice, using Proposition \ref{prp:bb polynomials and loop homology of zk}.
\end{proof}
\section{Homotopy groups of partial quotients and toric orbifolds}
\label{sec:partquoandorbi}

The action of tori on moment-angle complexes give rise to spaces which are of much interest in toric topology. In this section, we use Corollary ~\ref{crl:MACenumhtpy} to enumerate the homotopy groups of quotients of certain torus actions. The first case are partial quotients, which correspond to free actions. This class includes \emph{quasitoric manifolds} \cite[\S 7.3]{BP15} \cite{dj}, topological analogues of smooth projective toric varieties. The second case are \emph{toric orbifolds} (i.e. toric varieties over $\CC$ with finite quotient singularities), which correspond to certain almost free actions. 

\subsection{Partial quotients}

Recall that the standard action $T^m=(S^1)^m\curvearrowright (D^2)^m$ induces an action $H\curvearrowright \ZK=(D^2,S^1)^\K$ for any subgroup $H\subset T^m$.
\begin{lemma}
    \label{lem:actionfibsplit}
    Let $\K$ be a simplicial complex on $[m]$.
    Let $H$ be a closed, connected subgroup of $T^m$ which acts on $\ZK$. There is a homotopy equivalence \[\Omega (EH\times_H \ZK) \simeq (S^1)^{m-n} \times \Omega \ZK,\] where $n=\dim H\leq m$.
\end{lemma}
\begin{proof}
    The action of $H$ on $\ZK$ implies there is a principal $H$-fibration $H \rightarrow \ZK \rightarrow EH\times_H \ZK$. Since $H$ is a closed, connected subgroup of $T^m$, $H \cong T^{m-n}$ for some $n\leq m$. Therefore, there is a homotopy fibration $\Omega (EH\times_H \ZK) \xrightarrow{r} T^{m-n} \rightarrow \ZK$. Since $\ZK$ is $2$-connected \cite[Proposition
4.3.5 (a)]{BP15}, $EH \times_H \ZK$ is simply connected, and $r$ induces an isomorphism on $\pi_1$. Each generator of a $\mathbb{Z}$ summand in $\pi_1(\Omega (EH\times_H \ZK))$ is in the Hurewicz image. Using the $H$-space structure on $\Omega (EH\times_H \ZK)$, these maps can be multiplied to obtain a map $s:T^{m-n} \rightarrow \Omega (EH \times_H \ZK)$ which is a right homotopy inverse for $r$. Hence, we obtain the claimed homotopy equivalences.
\end{proof}

If the restricted action $H\curvearrowright \ZK$ is free, the quotient space $\ZK/H$ is called a \emph{partial quotient}. In this case there is a homotopy equivalence $ET^{m-n} \times_{T^{m-n}} \ZK \simeq \ZK/T^{m-n}$. Therefore, we can apply Lemma ~\ref{lem:actionfibsplit} to obtain the following.

\begin{proposition}
\label{prop:partquotentdecomp}
    Let $\K$ be a simplicial complex on $[m]$, and let $\ZK/T^{m-n}$ be a partial quotient. Then there is a homotopy equivalence
    $$\Omega (\ZK/T^{m-n})\simeq T^{m-n}\times \OZK.$$
    In particular, $\pi_2(\ZK/T^{m-n})\simeq\ZZ^{m-n}$, and $\pi_N(\ZK/T^{m-n})\cong\pi_N(\ZK)$ for $N\geq 3$. \qed
\end{proposition}
Combined with Corollary~\ref{crl:MACenumhtpy}, for large enough primes $p$ this result provides a description of the $p$-primary components of $\pi_*(\ZK/T^{m-n})$ in terms of homotopy groups of spheres.

The following follows immediately from Theorem \ref{thm:anick for djk} and Proposition ~\ref{prop:partquotentdecomp}.
\begin{corollary}
\label{cor:anick for partial quotient}
Let $\ZK/T^{m-n}$ be a partial quotient. Then $\Omega (\ZK/T^{m-n})\in\prod\mathcal{P}$ after localising away from a finite set of primes. In particular,
Anick's conjecture holds for $\ZK/T^{m-n}$.\qed
\end{corollary}

Note that every quasitoric manifold is equivariantly homeomorphic to a partial quotient by \cite[Proposition 7.3.13]{BP15}. Simply connected smooth toric varieties over $\mathbb{C}$ are equivariantly homotopy equivalent to partial quotients, see \cite[Proposition 4.1]{franz_toric}. In general, smooth toric varieties are covered by partial quotients, see Proposition \ref{prp:smooth toric} below.

\subsection{A local homotopy equivalence for almost free actions}
We will use the following lemma to obtain information about homotopy groups of almost free quotients. The argument is well known to experts. Our proof follows \cite[Proposition 1.6]{ustinovskii}. See also \cite[Lemma 4.2]{franz_toric} for an alternative argument.
\begin{lemma}
\label{lemma:quotient-localisation}
Let $Z$ be a simply connected $G$-CW-complex, where $G\simeq(S^1)^m$ is a compact torus. Let $H\subset G$ be a closed subgroup such that $H\curvearrowright Z$ almost freely, i.e. for any $z\in Z$ the stabilizer $H_z=\{h\in H:h.z=z\}$ is finite. Denote by $T=H^\circ$ the maximal connected subgroup of $H$, and consider the set of primes 
$$P=\{\text{prime divisors of }|\pi_0(H)|\}\cup\bigcup_{z\in Z}\{\text{prime divisors of }|H_z|\}.$$
Then
\begin{enumerate}
    \item $\pi_1(EH\times_H Z)\cong\pi_0(H)$, and the natural map $p:ET\times_T Z\to EH\times_H Z$ is the universal covering.
    \item The natural map $H_*(ET\times_T Z;\ZZ[1/P])\to H_*(Z/H;\ZZ[1/P])$ is an isomorphism.
    \item If $Z/H$ is simply connected, then $ET\times_T Z\simeq Z/H$ localised away from $P$.
\end{enumerate}
\end{lemma}
\begin{proof}
We use the space $EG$ as a model for both $EH$ and $ET$. By doing so, we identify the map $p:ET\times_T Z\to EH\times_H Z$ with the quotient map $EG\times_T Z\to EG\times_H Z$ by the free action of $H/T\cong \pi_0(H)$. In particular, $p$ is a covering map.

Since $\pi_1(Z)=0$, the exact sequence of homotopy groups for $Z\to EH\times_H Z\to BH$ gives the isomorphism $\pi_1(EH\times_H Z)\cong\pi_0(H)$. By the same argument, $\pi_1(ET\times_T Z)\cong\pi_0(T)=0$. Hence the covering map $p$ is universal, implying (1).

The statement (3) follows from (2) by Whitehead's theorem, since $\pi_1(ET\times_T Z)=0$ by (1) and $\pi_1(Z/H)=0$ by assumption.

Now we prove (2). The map $ET\times_T Z\to Z/H$ factors as the composition
$$ET\times_TZ=EG\times_T Z\overset{p}\longrightarrow EG\times_H Z=EH\times_H Z\overset{q}\longrightarrow Z/H,$$
where $p$ is a covering map with fiber $\pi_0(H)$.
Since $|\pi_0(H)|$ is invertible in $\ZZ[1/P]$, the transfer map produces an isomorphism $p_*:H_*(EG\times_T Z;\ZZ[1/P])^{\pi_0(H)}\to H_*(EG\times_H Z;\ZZ[1/P])$. The action of $\pi_0(H)=H/T$ on $H_*(EG\times_T Z)$ extends to the action of the connected group $G/T$, hence it is trivial. It follows that $H_*(EG\times_T Z;\ZZ[1/P])\to H_*(EG\times_H Z;\ZZ[1/P])$ is an isomorphism.

It remains to show that $q_*:H_*(EH\times_H Z;\ZZ[1/P])\to H_*(Z/H;\ZZ[1/P])$ is an isomorphism. For a point $[z]\in Z/H$ represented by $z\in Z$, $q^{-1}([z])=EH\times_H (Hz)\cong EH\times_H (H/H_z)\simeq BH_z$, which is a classifying space for the finite abelian group $H_z$. Since $|H_z|$ is invertible in $\ZZ[1/P]$, we have $\H_*(BH_z;\ZZ[1/P])=0$. It follows that $q$ has $\ZZ[1/P]$-acyclic fibers, hence $q_*$ is an isomorphism by the Vietoris--Begle theorem \cite[9.15]{spanier} (applied to the map $\mathrm{sk}_n(EH\times_H Z)\to Z/H$ for sufficiently large $n$).
\end{proof}

\subsection{Universal coverings of toric varieties}
In the rest of this section we study the homotopy of \emph{toric orbifolds} $X_\Sigma$, i.e. normal toric varieties over $\mathbb{C}$ with finite quotient singularities. By \cite[Theorem 11.4.8]{cls}, $X_\Sigma$ is an orbifold if and only if the corresponding fan $\Sigma$ is simplicial. We do not require $X_\Sigma$ to be smooth or complete. Our references for the geometry and topology of toric varieties are \cite{cls}, \cite[Ch. 5]{BP15} and \cite{franz_toric}.

We first note that the universal covering of any normal toric variety is itself a toric variety (surprisingly, we could not find a proof in literature). Let $\Sigma$ be a rational polyhedral fan in a lattice $N$. Its \emph{support} is the closed subset $|\Sigma|:=\bigcup_{\sigma\in\Sigma}\sigma\subset N\otimes_\ZZ\RR$.
Consider the sublattice
$$N_\Sigma:=\spn_\ZZ(N\cap |\Sigma|)\subset N.$$
\begin{proposition}
\label{prp:orbi universal cover}
Let $\Sigma$ be a fan in a lattice $N$, and let $X_\Sigma$ be the corresponding toric variety. Let $\widehat{\Sigma}$ be the fan $\Sigma$ considered in the lattice $N_\Sigma$. Then
\begin{enumerate}
    \item $\pi_1(X_\Sigma)\cong N/N_\Sigma$.
    \item The map $X_{\widehat{\Sigma}}\to X_\Sigma$, induced by the inclusion $N_\Sigma\to N$, is homotopy equivalent to the universal covering of $X_\Sigma$.
    \item $\Omega X_\Sigma\simeq \Omega X_{\widehat{\Sigma}}\times (N/N_\Sigma).$
\end{enumerate}
\end{proposition}
\begin{proof}
Statement (1) is \cite[Theorem 12.1.10]{cls}, and statement (3) is immediate from (2).

To prove (2), consider the intermediate lattice $N'=N\cap \spn_\RR(a_1,\dots,a_m)\subset N$, and let $\Sigma'$ be the fan $\Sigma$ considered in the lattice $N'$.

Since $N/N'$ is torsion free, there is an isomorphism of lattices $N\simeq N'\times N''$.  By \cite[Proposition 3.3.9]{cls}, this splitting gives an isomorphism  of toric varieties $X_\Sigma\simeq X_{\Sigma'}\times (N''\otimes_\ZZ \CC^\times)$, where $\Sigma'$ is the fan $\Sigma$ considered in the lattice $N'$. The space $N''\otimes_\ZZ \CC^\times \simeq (S^1)^r$ has a contractible universal covering, so we can assume that $N=N'$. Since $N\otimes_\ZZ\RR$ is spanned by $a_1,\dots,a_m$, $N_\Sigma\subset N$ is a sublattice of finite index.

By \cite[Proposition 3.3.7, Example 5.0.13]{cls}, the inclusion of lattices $N_\Sigma\hookrightarrow N$ induces a map of toric varieties $X_{\widehat{\Sigma}}\to X_{\Sigma}$ which is a geometric quotient by the action of the finite group $G=N/N_\Sigma$. So $X_\Sigma\cong X_{\widehat{\Sigma}}/G$. By (1), $X_{\widehat{\Sigma}}$ is simply connected. Also, by \cite{armstrong}, $\pi_1(X_{\widehat{\Sigma}}/G)\cong G/H$, where $H\subset G$ is the normal closure of stabilisers of action. It follows that $H=\{1\}$, i.e. the action is free. Hence $X_{\widehat{\Sigma}}\to X_\Sigma$ is a covering map. 
\end{proof}
\begin{corollary}
\label{crl:orbifold simply connected}
If $\Sigma$ has a cone of dimension $n$, or $N=\spn_\ZZ(a_1,\dots,a_m)$, then $X_\Sigma$ is simply connected. In particular, this holds if $X_\Sigma$ is complete.\qed
\end{corollary}
\subsection{Decompositions for simply connected toric orbifolds}
From now on, $X_\Sigma$ is a simply connected toric orbifold, i.e. $\Sigma$ is a simplicial fan such that $N=N_\Sigma$. We start with a model for $X_\Sigma$ up to an equivariant homotopy equivalence.

Let $a_1,\dots,a_m\in N$ be the primitive vectors on the rays of $\Sigma$. We obtain a simplicial complex $$\K=\{I\subseteq[m]:\{a_i:i\in I\}\text{ span a cone of }\Sigma\}$$
without ghost vertices, 
and a linear map $A:\ZZ^m\to N,$ $e_i\mapsto a_i$, such that the vectors $a_i\in N$ are primitive. In fact, the simplicial fan $\Sigma$ is determined by the pair $(\K,A)$, and this data corresponds to a fan if and only if the following conditions hold. For $I\in\K$, consider the closed convex cone
$$\sigma_I:=\mathbb{R}_{\geq 0}\{a_i:i\in I\}\subseteq N\otimes_\ZZ\RR;$$
then $(\K,A)$ form a simplicial fan if and only if
\begin{itemize}
    \item For any $I\in\K$, $\{a_i:i\in I\}\subseteq N$ are linearly independent (i.e. $\dim\sigma_I=|I|+1$);
    \item For any $I,J\in\K$, $\sigma_I\cap\sigma_J=\sigma_{I\cap J}$.
\end{itemize}
By assumption, $N=\spn_\ZZ(N\cap|\Sigma|)$; in particular, $a_1,\dots,a_m$ span $N\otimes_\ZZ\RR$.
Now consider
$$H_\Sigma:=\Ker(\exp A:(S^1)^m\to N\otimes_\ZZ S^1)\subseteq (S^1)^m.$$ The subgroup $H_\Sigma$ acts on $\ZK$, so we can consider the quotient $\ZK/H_\Sigma$. While the quotient $\ZK/H_\Sigma$ is not itself a toric orbifold, it is a homotopy theoretic model of one by the following result of Franz. In fact, $\ZK/H_\Sigma$ is a ``compact version'' of Cox's quotient construction for $X_\Sigma$, see \cite[Theorem 5.4.5(b)]{BP15}, where Cox's construction is obtained by replacing the pair $(D^2,S^1)$ with $(\CC,\CC^\times)$.

\begin{theorem}[{\cite[Proposition 4.1]{franz_toric}}]
\label{thm:toricorbifoldquotient}
    If $X_\Sigma$ is a toric orbifold associated to a rational simplicial fan $\Sigma = (\K,A)$, then $X_\Sigma$ is equivariantly homotopy equivalent to $\ZK/H_\Sigma$. \qed
\end{theorem}
In Theorem \ref{thm:orbifold p-local fibration}, we will give a decomposition of $\Omega(\ZK/H_\Sigma)$ by applying Lemma \ref{lemma:quotient-localisation} to the action $H_\Sigma\curvearrowright \ZK$, and then applying Lemma ~\ref{lem:actionfibsplit} to the Borel construction $ET_\Sigma\times_{T_\Sigma} \ZK$, where $T_\Sigma:=H_\Sigma^\circ\simeq T^{m-n}$ is the maximal connected subgroup of $H_\Sigma$. Note that $\K$ has no ghost vertices, hence $\ZK$ is simply connected by \cite[Proposition
4.3.5 (a)]{BP15}.

\medskip

To go further, we describe $H_\Sigma$ and the action $H_\Sigma \curvearrowright \ZK$ in more detail. Consider the lattice $$N'_\Sigma:=\spn_\ZZ(a_1,\dots,a_m)=\Img(A)\subseteq N.$$
\begin{lemma}
\label{lmm:pi0(H)}
There are isomorphisms $\pi_1(EH_\Sigma\times_{H_\Sigma}\ZK)\cong\pi_0(H_\Sigma)\cong N/N'_\Sigma$.
\end{lemma}
\begin{proof}
Since the rays of $\Sigma$ span $N\otimes_\ZZ\RR$, the map $\exp A:(S^1)^m\to N\otimes_\ZZ S^1$ is onto. For the fibration $H_\Sigma\to (S^1)^m\to N\otimes_\ZZ S^1$ of topological groups, the exact sequence of homotopy groups
$$\pi_1((S^1)^m)\overset{(\exp A)_*}\longrightarrow\pi_1(N\otimes_\ZZ S^1)\to \pi_0(H_\Sigma)\to \pi_0((S^1)^m)$$
is identified with
$$\ZZ^m\overset{A}\longrightarrow N\to\pi_0(H_\Sigma)\to 0.$$
Hence, $\pi_0(H_\Sigma)\cong\Coker A=N/N'_\Sigma$. The isomorphism $\pi_1(EH_\Sigma\times_{H_{\Sigma}}\ZK)\cong\pi_0(H_\Sigma)$ follows from part (1) of Lemma \ref{lemma:quotient-localisation}.
\end{proof}

Next, we describe the stabilisers of the action $H_\Sigma\curvearrowright\ZK$. Define
\[P'_\Sigma:=\bigcup_{I\in\K}\{\text{prime divisors of }|\Tors N/N_I|\},\quad N_I:=\spn_\ZZ(a_i:i\in I)\subseteq N.\] Since the groups $N/N_I$ are finitely generated, $P'_\Sigma$ is a finite set of primes.
\begin{proposition}
\label{prp:orbifold stabilisers compact}
For each simplex $I\in\K$ consider the composite map
$$e_I=\exp(A|_I):(S^1)^I\hookrightarrow (S^1)^m\overset{\exp A}\longrightarrow N\otimes_\ZZ S^1$$
and the group $H_\Sigma(I):=\Ker(e_I)\subseteq H_\Sigma.$ Then
\begin{enumerate}
    \item The set of stabilizers $\{(H_\Sigma)_z:z\in \ZK\}$ of the action $H_\Sigma\curvearrowright\ZK$ coincides with the set $\{H_\Sigma(I):I\in\K\}$.
    \item For any $I\in\K$, $H_\Sigma(I)\cong\Tors N/N_I$. In particular, $H_\Sigma\curvearrowright\ZK$ is almost free.
    \item $P'_\Sigma=\bigcup_{z\in\ZK}\{\text{prime divisors of }|(H_\Sigma)_z|\}$.
\end{enumerate}
\end{proposition}
\begin{proof}
For a point $z=(z_1,\dots,z_m)\in(D^2)^{\times m}$, let $\omega(z)=\{i\in[m]:z_i=0\}$. Then we have $(H_\Sigma)_z=H_\Sigma(\omega(z))$; the proof is similar to \cite[Proposition 5.4.6(a)]{BP15}. By definition, we have $\omega(z)\in\K$ for $z\in\ZK$. Conversely, for every $I\in\K$ there is $z\in\ZK$ such that $\omega(z)=I$. This proves (1). For (2), note that the set $\{a_i:i\in I\}\subseteq N\simeq\ZZ^n$ is linearly independent, so there is a short exact sequence
$$0\to \ZZ^I\overset{A|_I}\longrightarrow N\to N/N_I\to 0.$$
Applying the functor $(-)\otimes_\ZZ S^1$, we obtain an exact sequence
$$0\to\Tor^\ZZ_1(N/N_I,S^1)\to (S^1)^I\overset{e_I}\longrightarrow N\otimes_\ZZ S^1.$$
Hence $H_\Sigma(I)=\Ker(e_I)\cong\Tor^\ZZ_1(N/N_I,S^1)\cong\Tors N/N_I$ is a finite group, so $H_\Sigma\curvearrowright \ZK$ is almost free.
\end{proof}

Consider the set of primes
$P_\Sigma := \{\text{prime divisors of }|\Tors N/N'_\Sigma|\}\cup P'_\Sigma.$
\begin{lemma}
\label{lmm:PSigma=P'Sigma}
If $\Sigma$ has a cone of dimension $n$, then $P_\Sigma=P'_\Sigma$. 
\end{lemma}
\begin{proof}
Suppose $\sigma_I$ is of dimension $n$. Then $N_I=\spn_\ZZ(a_i:i\in I)\subseteq N$ is a sublattice of finite index. Also, $N_I\subseteq N'_\Sigma$. Hence $N/N'_\Sigma$ is a quotient group of the finite group $N/N_I$, so $|\Tors N/N'_\Sigma|$ divides $|\Tors N/N_I|$.
\end{proof}

\begin{proposition}
\label{prp:orbifold is locally a borel construction}
Let $\Sigma=(\K,A)$ be a simplicial fan in $N\simeq\ZZ^n$. Consider the natural map $f:ET_\Sigma\times_{T_\Sigma}\ZK\to EH_\Sigma\times_{H_\Sigma}\ZK\to  \ZK/H_\Sigma$. Then
\begin{enumerate}
    \item $f$ induces an isomorphism on homology with coefficients in $\ZZ[1/P_\Sigma]$.
    \item Suppose that $N=N_\Sigma$. Then $f$ is a homotopy equivalence localised away from $P_\Sigma$.
\end{enumerate}
\end{proposition}
\begin{proof}
We apply Lemma \ref{lemma:quotient-localisation} to the almost free action $H_\Sigma\subset (S^1)^m\curvearrowright \ZK$.  It remains to show that
$$P_\Sigma=\{\text{prime divisors of }|\pi_0(H_\Sigma)|\}\cup\bigcup_{z\in\ZK}\{\text{prime divisors of }|(H_\Sigma)_z|\},$$
and this follows from Lemma \ref{lmm:pi0(H)} and Proposition \ref{prp:orbifold stabilisers compact}. Statement (2) follows from part (3) of Lemma \ref{lemma:quotient-localisation}, since $\pi_1(\ZK/H_\Sigma)=N/N_\Sigma$ by Proposition \ref{prp:orbi universal cover}.
\end{proof}

\begin{theorem}
\label{thm:orbifold p-local fibration}
    Let $X_\Sigma$ be a toric orbifold corresponding to a simplicial fan $\Sigma=(\K,A)$ in a lattice $N\simeq\ZZ^n$.
    Suppose that $N=N_\Sigma$, i.e. $X_\Sigma$ is simply connected. 
Then, localised away from the finite set of primes
    $$P_\Sigma=\{\text{prime divisors of }|\Tors N/N'_\Sigma|\}\cup \bigcup_{I\in\K}\{\text{prime divisors of }|\Tors N/N_I|\},$$ there is a homotopy equivalence $\Omega X_\Sigma \simeq \Omega \ZK \times T^{m-n}$.
\end{theorem}
\begin{proof}
By Theorem ~\ref{thm:toricorbifoldquotient}, $X_\Sigma$ is homotopy equivalent to $\ZK/H_\Sigma$. Localise away from primes in $P_\Sigma$. Proposition ~\ref{prp:orbifold is locally a borel construction} implies that $ET_\Sigma \times_{T_\Sigma} \ZK \simeq \ZK/H_{\Sigma}$, where $T_\Sigma\subset T^m$ is a subtorus, so $T_\Sigma\simeq T^{m-n}$. Hence, by Lemma ~\ref{lem:actionfibsplit}, there is a homotopy equivalence $\Omega X_\Sigma \simeq \Omega (\ZK/H_{\Sigma}) \simeq \Omega \ZK \times T^{m-n}$.
\end{proof}

\begin{corollary}
\label{crl:orbifold p-local homotopy type}
Let $X_\Sigma$ be a toric orbifold corresponding to a simplicial fan $\Sigma=(\K,A)$ in a lattice $N\simeq\ZZ^n$. Suppose that $N=N_\Sigma$, i.e. $X_\Sigma$ is simply connected.
Then, localised away from the finite set of primes $P=P_\Sigma\cup\tau(\OZK)\cup\{p:p<2m\}$,
$$\Omega X_\Sigma\simeq  T^{m-n}\times\prod_{n\geq 3}(\Omega S^n)^{\times D_n},$$
where the numbers $D_n$ are described in Corollary \ref{crl:MACenumhtpy}.
Moreover, $\pi_2(X_\Sigma)=H_2(X_\Sigma;\ZZ)$ and 
$$\pi_2(X_\Sigma)\otimes\ZZ[1/P]\simeq\ZZ^{m-n}\otimes\ZZ[1/P],\quad\pi_N(X_\Sigma)\otimes\ZZ[1/P]=\pi_N(\ZK)\otimes\ZZ[1/P],\quad N\geq 3.$$
\end{corollary}
\begin{proof}
Localise away from $P$. By Theorem ~\ref{thm:orbifold p-local fibration}, $\Omega X_\Sigma\simeq T^{m-n}\times\OZK$. Since $\ZK$ is $2$-connected and $X_\Sigma$ is $1$-connected, the rest follows from Corollary \ref{crl:MACenumhtpy} and the Hurewicz isomorphism.
\end{proof}

Combining Theorem ~\ref{thm:orbifold p-local fibration} with Theorem ~\ref{thm:anick for djk}, we obtain the following result.
\begin{corollary}
For any simply connected toric orbifold $X_\Sigma$, $\Omega X_\Sigma\in\prod\mathcal{P}$ after localising away from a finite set of primes. In particular, $\Omega X_\Sigma$ has FPT and $X_\Sigma$ satisfies Anick's conjecture.\qed
\end{corollary}
As a special case, we recover an integral loop space decomposition of any smooth toric variety. 
\begin{proposition}
\label{prp:smooth toric}
Let $X_\Sigma$ be a smooth toric variety of dimension $n$. Then $P_\Sigma=\varnothing$, and the universal cover of $X_\Sigma$ is homotopy equivalent to a partial quotient $\ZK/T_\Sigma$, $T_\Sigma\simeq T^{m-n}$. In particular, there is a homotopy equivalence
$$\Omega X_\Sigma\simeq T^{m-n}\times\OZK\times N/N_\Sigma.$$
\end{proposition}
\begin{proof}
By \cite[Theorem 3.1.19]{cls}, smooth toric varieties correspond to \emph{smooth fans}, i.e. rational fans $\Sigma$ such that each $N_I$ is a direct summand in $N$. It follows that $\Tors N/N_I=0$, so $P_\Sigma=\varnothing$ and the action $H_\Sigma\curvearrowright\ZK$ is free. Hence $\ZK/T_\Sigma\to\ZK/H_\Sigma\simeq X_\Sigma$ is a covering map. Since $\pi_1(\ZK/T_\Sigma)=0$, this is the universal covering. The rest follows from Propositions \ref{prop:partquotentdecomp} and \ref{prp:orbi universal cover}.
\end{proof}

\subsection{Example: toric surfaces}
Algebraic topology of quasitoric orbifolds of complex dimension $2$ (homotopy classification, structure of integral cohomology ring) was studied in detail by Fu, So and Song \cite{fss1,fss2,fss3}. Our results allow for $P$-local loop space decomposition of these spaces. We restrict to the case of toric varieties, though the argument works more generally.

Let $X=X_\Sigma$ be a complete toric variety of complex dimension $2$. Then $N\cap |\Sigma|=N$, so $X$ is simply connected by Proposition \ref{prp:orbi universal cover}. By \cite[Theorem 3.4.6]{cls}, $\Sigma$ is a complete fan in $N=\ZZ^2$, hence $\K=C_m$ is a boundary of an $m$-gon. Let $a_i=(x_i,y_i)\in \ZZ^2$ be the primitive vectors on the rays of $\Sigma$, where $1\leq i\leq m$ and we identify $m+1$ with $1$. 
By Theorem \ref{thm:toricorbifoldquotient}, $X_\Sigma$ is homotopy equivalent to $\mathcal{Z}_{C_m}/H_\Sigma$, where $$H_\Sigma=\{(t_1,\dots,t_m)\in T^m:\prod_{i=1}^m t_i^{x_i}=\prod_{i=1}^mt_i^{y_i}=1\}.$$
For $i,j\in [m]$, consider the positive integers $m_{ij}:=|x_iy_j-x_jy_i|$, and define
$$d_1:=gcd(\{x_i,y_i:i\in[m]\}),\quad d_2:=gcd(\{m_{ij}:1\leq i<j\leq m\})/d_1.$$
\begin{proposition}
With notation as above, the following hold:
\begin{enumerate}
    \item $\pi_0(H_\Sigma)=\ZZ/d_1\oplus\ZZ/d_2$. In particular, the subgroup $H_\Sigma\subset T^m$ is connected if and only if $gcd(\{m_{ij}:1\leq i<j\leq m\})=1$.
    \item Suppose that $H_\Sigma$ is connected. Localised away from the finite set of primes
    $$P=\bigcup_{i=1}^m\{\text{prime divisors of }|x_iy_{i+1}-x_{i+1}y_i|\},$$
    there is a homotopy equivalence $\Omega X\simeq T^{m-2}\times\prod_{n\geq 3}(\Omega S^n)^{\times D_n}$, where the numbers $D_n$ depend only on $m$ and are determined by the identity
$$\prod_{n\geq 3}(1-t^{n-1})^{D_n}=
\begin{cases}
1-t^4,&m=3;\\
(1-(m-2)t+t^2)(1+t)^{m-2},&
m\geq 4.
\end{cases}$$
\end{enumerate}
\end{proposition}
\begin{proof}
We have $N'_\Sigma=\spn_\ZZ(a_1,\dots,a_m)=\Img(\ZZ^m\to \ZZ^2,~e_i\mapsto a_i)$. By \cite[Proposition 8.1]{smith}, the Smith normal form of this map has the matrix $\mathrm{diag}(d_1,d_2)$, hence $\ZZ^2/N'_\Sigma\simeq\ZZ/d_1\oplus\ZZ/d_2$, and (1) follows from Lemma \ref{lmm:pi0(H)}.

Now we prove (2). By Theorem \ref{thm:toricorbifoldquotient}, $\Omega X\simeq\Omega\mathcal{Z}_{C_m}\times T^{m-2}$ localised away from $P_\Sigma$. By Lemma \ref{lmm:PSigma=P'Sigma}, $P_\Sigma=P'_\Sigma$. Now we check that $P'_\Sigma=P$. By definition, $P'_\Sigma$ is the set of prime divisors of $|N/N_I|$, $I\in\K$. Since $\K=C_m$, for $I\in\K$ we have either $I=\{i\}$ or $I=\{i,i+1\}$. In the first case, $N_I=\spn(a_i)$ is a direct summand of $N=\ZZ^2$, so $\Tors N/N_I=0$; in the second case, $N_I=\spn(a_i,a_{i+1})\subset\ZZ^2$ is a sublattice of index $|x_iy_{i+1}-x_{i+1}y_i|$, hence $|\Tors N/N_I|=|x_iy_{i+1}-x_{i+1}y_i|$. It follows that $P_\Sigma=P$.

Finally, we describe the homotopy type of $\Omega\mathcal{Z}_{C_m}$ in terms of products on loops of spheres. For $m=3$, we have $\Omega\mathcal{Z}_{C_3}=\Omega(\partial D^6)=\Omega S^5$. Otherwise $C_m$ is a flag complex, hence the asserted loop space decomposition follows from \cite[Theorem 1.2]{vylegzhanin} and $h_{C_m}(t)=1+(m-2)t+t^2$. 
\end{proof}

\begin{remark}
A result of McGavran \cite[Theorem 3.6]{mcgavran} implies that $\mathcal{Z}_{C_m}$ is homeomorphic to a connected sum of sphere products,
$$\mathcal{Z}_{C_m}\cong \#_{k=3}^{m-2}(S^k\times S^{m-k+2})^{\# (k-2)\binom{m-2}{k-1}}.$$ The homotopy type of $\Omega \mathcal{Z}_{C_m}$ can be recovered via different methods by \cite[Example 3.1]{BT14}.
\end{remark}

The case of \emph{non-complete} toric surfaces having at least one two-dimensional cone is similar.
In this case, $\K$ is a proper subcomplex of an $m$-gon (i.e. a disjoint union of path graphs), hence $\ZK$ is homotopy equivalent to a wedge of spheres (see e.g. \cite[Theorem 8.2.1]{BP15} or \cite[Theorem 4.6]{GPTW16}). If $\K$ has $m$ vertices and $e<m$ edges (i.e. $\Sigma$ has $e$ two-dimensional cones), $h_\K(t)=1+(m-2)t+(e-m+1)t^2$, so the numbers $D_n$ satisfy the identity \[\prod\limits_{n\geq 3}(1-t^{n-1})^{D_n}=(1-(m-2)t+(e-m+1)t^2)(1+t)^{m-2}.\]

Finally, if $\K$ has only one-dimensional cones, then $X_\Sigma$ is a smooth toric variety (possibly with nontrivial fundamental group). The homotopy type of $\Omega X_\Sigma$ is described in Proposition \ref{prp:smooth toric}. 

\section{Loop spaces of simply connected polyhedral products}
\label{sec:LoopCaa}

In this section, we reduce the study of many properties of general looped polyhedral products down to the moment-angle complex case.

\subsection*{Setup}

We first state some preliminary results we will use in our method. First, we state a loop space decomposition for polyhedral products of the form $\caa^\K$ which will be crucial to reduce the study of $\Omega \caa^\K$ to that of $\Omega \ZK$. 

\begin{theorem}[{\cite[Proposition 4.3]{PT19}}]
\label{thm:gendecompuxa}
    Let $\K$ be a simplicial complex on $[m]$ and let $(X_1,A_1),$ $\dots,$ $(X_m,A_m)$ be $CW$-pairs. For any vertex $i \in \K$, there is a homotopy equivalence \[\Omega \uxa^\K \simeq \Omega X_i \times \Omega \uxa^{\K \setminus i} \times \Omega \Sigma(F \wedge G_i),\] where $F$ is the homotopy fibre of $\uxa^{link_\K(i)} \rightarrow \uxa^{\K \setminus i}$ (with $link_\K(i)$ and $\K\setminus i$ being considered as simplicial complexes on $[m]\setminus i$), and $G_i$ is the homotopy fibre of $A_i \rightarrow X_i$. \qed 
\end{theorem}

A special case of Theorem ~\ref{thm:gendecompuxa} will be used.

\begin{lemma}
    \label{lmm:replacing A with Y}
    Let $\K$ be a simplicial complex on $[m]$ and let $A_1,\dots,A_m$ be pointed $CW$-complexes. For a fixed $i \in [m]$, define $Y_i = S^1$ and $Y_j = A_j$ for all $i \neq j$. There are homotopy equivalences \[\Omega \caa^\K \simeq \Omega \caa^{\K \setminus i} \times \Omega \Sigma (F \wedge A_i),\]\[\Omega \cyy^\K \simeq \Omega \caa^{\K \setminus i} \times \Omega \Sigma (F \wedge S^1),\] where $F$ is the homotopy fibre of $\caa^{link_\K(i)} \rightarrow \caa^{\K \setminus i}$ with $link_\K(i)$ and $\K\setminus i$ being considered as simplicial complexes on the vertex set $[m]\setminus i$. In particular, $\Omega \caa^{\K \setminus i}$ and $\Omega \Sigma^2 F$ retract off $\Omega \cyy^\K$.
\end{lemma}
\begin{proof}
    By definition of each $Y_j$, $\caa^{\K \setminus i} = \cyy^{\K \setminus i}$, and $\caa^{link_\K(i)} = \cyy^{link_\K(i)}$. The result follows by applying Theorem ~\ref{thm:gendecompuxa} to $\caa^\K$ and $\cyy^{\K}$.
\end{proof}

We will require a technical result on the homotopy fibre $F$ appearing in Lemma ~\ref{lmm:replacing A with Y}.
\begin{lemma}
    \label{lem:FhasZ}
    Let $\K$ be a simplicial complex on $[m]$. Suppose there exists a vertex $i$ such that $\K \neq (\K \setminus i) * i$, and denote by $F$ the homotopy fibre of $\caa^{link_\K(i)} \rightarrow \caa^{\K \setminus i}$. If each $\widetilde{H}^*(A_j)$ contains a $\mathbb{Z}$-summand, then $\widetilde{H}^*(F)$ contains a $\mathbb{Z}$-summand.
\end{lemma}
\begin{proof}

    Since $\K \neq (\K \setminus i) * i$, $link_\K(i) \neq K \setminus i$. If $link_\K(i) = \varnothing$, then by definition of $\caa^{link_{\K}(i)}$, $F$ is the homotopy fibre of $\prod_{j\neq i}A_j \rightarrow \caa^{\K \setminus i}$. By \cite[Corollary 3.3]{GT13}, this map is null homotopic, and so $F \simeq \prod_{j\neq i}A_j \times \Omega \caa^{\K \setminus i}$. The result then follows from the K\"unneth theorem.
    
    Now suppose $link_\K(i) \neq \varnothing$. Let $\sigma = \{i_1,\cdots,i_k\}$ be a minimal missing face of $link_\K(i)$ such that $\sigma$ is a face of $K \setminus i$. By definition, $\partial \sigma$ is a full subcomplex of $link_{\K}(i)$, and so by \cite[Lemma 2.2.3]{DS07}, the map $g:\caa^{\partial \sigma} \rightarrow \caa^{link_\K(i)}$ has a left homotopy inverse (here $\partial\sigma$ is considered as a simplicial complex on the vertex set $\sigma$). The construction of the polyhedral product implies there is a commutative diagram \[\begin{tikzcd}
	{\caa^{\partial \sigma}} & {\caa^{\sigma}} \\
	{\caa^{link_\K(i)}} & {\caa^{\K \setminus i}.}
	\arrow[from=1-1, to=1-2]
	\arrow["g", from=1-1, to=2-1]
	\arrow[from=1-2, to=2-2]
	\arrow["f", from=2-1, to=2-2]
\end{tikzcd}\] By definition of the polyhedral product, $\caa^{\sigma}$ is contractible, and so the composite $f \circ g$ is null homotopic. Hence, there is a map $\lambda:\caa^{\partial \sigma} \rightarrow F$ satisfying the homotopy commutative diagram \[\begin{tikzcd}
	& F \\
	{\caa^{\partial \sigma}} & {\caa^{link_\K(i)}.}
	\arrow[from=1-2, to=2-2]
	\arrow["\lambda", from=2-1, to=1-2]
	\arrow["g", from=2-1, to=2-2]
\end{tikzcd}\] The left homotopy inverse for $g$ implies $\lambda$ has a left homotopy inverse. By \cite[Proposition 2.3]{GT13}, there is a homotopy equivalence  $\caa^{\partial \sigma} \simeq \Sigma^{k-1}(A_{i_1} \wedge \cdots \wedge A_{i_k})$. Since each $\widetilde{H}^*(A_i)$ contains a $\mathbb{Z}$-summand, the reduced K\"unneth theorem implies $\H_*(\caa^{\partial\sigma})$ contains a $\mathbb{Z}$-summand. Hence, the left homotopy inverse for $\lambda$ implies $\widetilde{H}^*(F)$ contains a $\mathbb{Z}$-summand.
\end{proof}

Next, we state a homotopy fibration which splits after looping. This reduces the study of $\Omega \uxa^\K$ to polyhedral products of the form $\Omega \cgg^\K$.

\begin{theorem}[{\cite[Theorem 2.1]{HST19}}]
    \label{thm:uxasplitfib}
    Let $\K$ be a simplicial complex on $[m]$ and let $(X_1,A_1),$ $\dots,$ $(X_m,A_m)$ be $CW$-pairs. Denote by $G_i$ the homotopy fibre of $A_i \rightarrow X_i$ for all $i$. There is a homotopy fibration \[\cgg^\K \rightarrow \uxa^\K \rightarrow \prod\limits_{i=1}^m X_i\] which splits after looping, giving a homotopy equivalence \[\Omega \uxa^\K \simeq \prod\limits_{i=1}^m \Omega X_i \times \Omega \cgg^\K.\]
    In particular,
    \[
    \Omega\ux^\K\simeq\prod_{i=1}^m\Omega X_i\times \Omega\clxx^\K,\quad
    \Omega\DJ(\K)\simeq (S^1)^m\times\OZK.\eqno\qed 
    \]
\end{theorem}

\subsection*{Torsion and the Steenrod algebra}
The first properties we consider are torsion in the loop homology of polyhedral products, and the action of the Steenrod algebra on the loop cohomology. Recall that, for a space $X$, the set of primes $p$ appearing as $p$-torsion in $H_*(X;\ZZ)$ is denoted by $\tau(X)$.  A space $X$ has FPT if $\tau(X)$ is a finite set. For a space $X$, let $s(X)$ be the set of primes $p$ such that the mod-$p$ Steenrod algebra acts non-trivially on $H^*(X;\FF_p)$. We first prove a preliminary result about the torsion and action of the Steenrod algebra on the product and smash product of spaces.

\begin{lemma}
    \label{lem:actiononproduct} Let $X$ and $Y$ be spaces. Then \begin{enumerate}
        \item $\tau(X \times Y) = \tau(X) \cup \tau(Y)$,
        \item $\tau(X \wedge Y) \subseteq \tau(X) \cup \tau(Y)$,
        \item $\tau(\Omega \Sigma X) = \tau(X)$,
        \item $s(X \times Y) = s(X) \cup s(Y)$,
        \item $s(X \wedge Y) \subseteq s(X) \cup s(Y)$.
        \item $s(\Omega  \Sigma X) = s(X)$.
    \end{enumerate}
    If both $\H^*(X)$ and $\H^*(Y)$ contain a $\ZZ$-summand, then (2) and (5) are equalities.
    \end{lemma}
\begin{proof}
    The results for $(1)$ and $(2)$ follow from the K\"unneth theorem and $(3)$ and $(6)$ follow from the James splitting $\Sigma \Omega \Sigma X \simeq \bigvee_{k \geq 1} \Sigma X^{\wedge k}$. (James splitting holds for connected $X$; in the disconnected case, note that $\Omega\Sigma (\bigsqcup_\alpha X_\alpha)\simeq C\times \Omega\Sigma(\bigvee_\alpha X_\alpha)$, where $C$ is discrete and $\bigvee_\alpha X_\alpha$ is connected.)

    We prove $(4)$ and $(5)$ follows from the same proof. Let $p \in s(X \times Y)$. Then either some reduced power $\mathcal{P}^i$ or the mod-$p$ Bockstein map $\beta$ acts non-trivially on $H^*(X\times Y;\FF_p)$. In the first case, let $i\geq 1$ be the least integer such that $\mathcal{P}^i(z)$ is non-trivial for some $z \in H^*(X \times Y;\FF_p)$. The K\"unneth theorem implies that $z = \sum_j a_j(x_j \otimes y_j)$. By the Cartan formula and the fact that $i$ is the smallest integer such that $\mathcal{P}^i$ is non-trivial, we obtain \[\mathcal{P}^i(z) = \sum_j a_j(\mathcal{P}^i(x_j) \otimes y_j + x_j \otimes \mathcal{P}^i(y_j)) \neq 0.\] Hence, there exists $\mathcal{P}^i(x_j) \neq 0$ or $\mathcal{P}^i(y_j) \neq 0$, and so $s(X \times Y) \subseteq s(X) \cup s(Y)$. The same argument holds for the mod-$p$ Bockstein since $\beta$ is a derivation.

    Now let $p \in s(X) \cup s(Y)$. Let $i\geq 1$ be the least integer such that $\mathcal{P}^i(x) \neq 0$ for some element $x \in H^*(X;\FF_p)$ or $\mathcal{P}^i(y) \neq 0$ for some $y \in H^*(Y;\FF_p)$. Suppose $\mathcal{P}^i(x) \neq 0$ for some $x \in H^*(X;\FF_p)$, and pick $y \neq 0 \in \widetilde{H}^*(Y;\FF_p)$. Note that in the smash product case, $y$ exists if $\widetilde{H}^*(Y;\mathbb{Z})$ contains a $\mathbb{Z}$ summand. The Cartan formula and choice of $i$ implies \[\mathcal{P}^i(x \otimes y) = \mathcal{P}^i(x) \otimes y + x \otimes \mathcal{P}^i(y) \neq 0.\] Hence, $s(X) \cup s(Y) \subseteq s(X \times Y)$, and the result follows.
\end{proof}
\begin{theorem}
\label{thm:caaFPT}
    Let $\K$ be a simplicial complex on $[m]$. Then
    \begin{gather*}\tau(\Omega \caa^\K) \subseteq \tau(\Omega \ZK) \cup \tau(A_1) \cup \dots \cup \tau(A_m),\\s(\Omega \caa^\K) \subseteq s(\Omega \ZK) \cup s(A_1) \cup \dots \cup s(A_m),
    \end{gather*}with equality if $\widetilde{H}^*(A_v)$ contains a $\mathbb{Z}$-summand and $\K \neq (\K \setminus v) *v$ for any vertex $v$. In particular, if each $A_i$ has FPT then $\Omega \caa^\K$ has FPT.
\end{theorem}
\begin{proof}
     If $\K = \Delta^{m-1}$, then $\caa^\K$ is contractible, so assume $\K \neq\Delta^{m-1}$. We prove the case of $\tau(\Omega \caa^{\K})$, and the proof for $s(\Omega \caa^{\K})$ is similar.
 
    We proceed by induction on $l$, where $l$ is the number of $A_i$'s which are not homotopy equivalent to $S^1$. If $l=0$, each $A_i \simeq S^1$. Therefore, $\caa^\K \simeq \ZK$ and the result holds trivially with equality.
    
     Suppose the result is true for $l-1$ and consider the case $l \geq 1$. Pick $i$ such that $A_i \not \simeq S^1$. Let $Y_i = S^1$ and $Y_j = A_j$ for all $j \neq i$.

     Let $S = \tau(\OZK) \cup \tau(Y_1) \cup \cdots \tau(Y_m)$, and $S' = \tau(\OZK) \cup \tau(A_1)\cup \cdots \cup \tau(A_m)$. Observe that since $Y_i = S^1$, $S' = S \cup \tau(A_i)$. By the inductive hypothesis, $\tau(\Omega\cyy^\K)\subseteq S$.

     On the other hand, by Lemma~\ref{lmm:replacing A with Y} and Lemma~\ref{lem:actiononproduct}
     \begin{alignat*}{2}
        \tau(\Omega\cyy^\K)&=\tau(\Omega\caa^{\K\setminus i})\cup\tau(\Omega\Sigma^2F)&&=\tau(\Omega\caa^{\K\setminus i})\cup\tau(F),\\
        \tau(\Omega\caa^\K)&=\tau(\Omega\caa^{\K\setminus i})\cup\tau(\Omega\Sigma(F\wedge A_i))&&=\tau(\Omega\caa^{\K\setminus i})\cup\tau(F\wedge A_i)\\&\subseteq \Big(\tau(\Omega\caa^{\K\setminus i})\cup\tau(F)\Big)\cup \tau(A_i) &&= \tau(\Omega\cyy^\K)\cup\tau(A_i).
     \end{alignat*}
    Therefore, $\tau(\Omega\caa^\K)\subseteq S\cup \tau(A_i)=S'$, as required.

    Now suppose that each $\H_*(A_v)$ contains a $\ZZ$-summand and $\K\neq(\K\setminus v)\ast v$ for any vertex $v$. Then $\tau(\Omega\cyy^\K)=S$ by the inductive hypothesis, and Lemma \ref{lem:FhasZ} implies that $\H_*(F)$ has a $\ZZ$-summand. By Lemma~\ref{lem:actiononproduct}, $\tau(F\wedge A_i)=\tau(F)\cup\tau(A_i)$. Hence all the inclusions above become equalities, and in this case $$\tau(\Omega\caa^\K)=\tau(\Omega\cyy^\K)\cup\tau(A_i)=S\cup\tau(A_i)=S'.$$
          
     Since $\OZK$ has FPT by Theorem ~\ref{thm:djk has fpt}, if each $A_i$ has FPT, the set $S'$ is finite. Hence, $\Omega \caa^\K$ has FPT.
\end{proof}

\begin{remark}
    The converse of Theorem ~\ref{thm:caaFPT} is not true. Let $S=\{p_s\}$ and $T=\{q_t\}$ be infinite, disjoint sets of odd primes. Let $\K$ be two disjoint points, $A_1 = \bigvee_{s \in S} P^{p_s}(p_s)$, and $A_2 = \bigvee_{t \in T} P^{p_t}(p_t)$. By definition, each $A_i$ does not have FPT, but $\caa^\K\simeq \Sigma A_1\wedge A_2$ is contractible. 
\end{remark}

Using Theorem ~\ref{thm:uxasplitfib}, we can give sufficient conditions for general looped polyhedral products of the form $\Omega \uxa^\K$ to have FPT.

\begin{theorem}
\label{thm:uxaFPT}
        Let $\K$ be a simplicial complex on $[m]$ and let $(X_1,A_1),\dots,(X_m,A_m)$ be $CW$-pairs. Denote by $G_i$ the homotopy fibre of $A_i \rightarrow X_i$ for all $i$. Then \begin{gather*}\tau(\Omega \uxa^\K) \subseteq \tau(\OZK) \cup\tau(\Omega X_1) \cup \dots \cup \tau(\Omega X_m) \cup \tau(G_1) \cup \dots \cup \tau(G_m),\\
        s(\Omega \uxa^\K) \subseteq s(\OZK) \cup s(\Omega X_1) \cup \dots \cup s(\Omega X_m) \cup s(G_1) \cup \dots \cup s(G_m),
        \end{gather*} with equality if each $\widetilde{H}^*(G_i)$ contains a $\mathbb{Z}$-summand, and $\K \neq (\K \setminus v) * v$ for any vertex $v$. In particular, if each $\Omega X_i$ and each $G_i$ has FPT, then $\Omega \uxa^\K$ has FPT.
\end{theorem}
\begin{proof}
    We prove the $\tau(\Omega \uxa^\K)$ case, and the $s(\Omega \uxa^\K)$ case follows from the same proof. By Theorem ~\ref{thm:uxasplitfib}, there is a homotopy equivalence \[\Omega \uxa^\K \simeq \prod\limits_{i=1}^m \Omega X_i \times \Omega \cgg^\K,\] and so Lemma ~\ref{lem:actiononproduct} implies \[\tau(\Omega \uxa^\K) = \tau(\Omega X_1) \cup \cdots \cup \tau(\Omega X_m) \cup \tau(\Omega \cgg^{\K}).\] By Theorem ~\ref{thm:caaFPT}, $\tau(\Omega \cgg^\K) \subseteq \tau(\OZK) \cup \tau(G_1) \cup \dots \cup \tau(G_m)$ with equality under the listed assumptions. Theorem ~\ref{thm:djk has fpt} implies $\tau(\OZK)$ is finite, and so if each $\tau(\Omega X_i)$ and each $\tau(G_i)$ is finite, then $\tau(\Omega\uxa^\K)$ is also finite.
\end{proof}

\begin{example}
\label{ex:torsionandSteenrod}
    Let $\K$ be a flag complex. By \cite[Corollary 7.3]{PT19}, $\Omega \mathcal{Z}_{\K}$ decomposes as a product of spheres and loops on spheres, and so $\tau(\OZK) = s(\Omega \mathcal{Z}_{\K}) = \varnothing$. Theorem ~\ref{thm:uxaFPT} implies that $\tau(\Omega \uxa^\K)$ and $s(\Omega \uxa^{\K})$ depends only on the ingredient spaces in this case.

    Now let $\K$ be the $6$-vertex triangulation of $\mathbb{R}P^2$. By \cite[Example 3.3]{GPTW16}, there is a homotopy equivalence \[\mathcal{Z}_{\K} \simeq W \vee \Sigma^7 \mathbb{R}P^2,\] where $W$ is a wedge of spheres. The James splitting implies that $\tau(\OZK) = s(\OZK)= \{2\}$. Therefore, Theorem ~\ref{thm:uxaFPT} implies that $H_*(\Omega \uxa^\K)$ contains $2$-torsion, and the mod-$2$ Steenrod algebra acts non-trivially on $H^*(\Omega \uxa^\K;\FF_2)$ if each $G_i$ has a $\mathbb{Z}$-summand in cohomology. An example of such a pair is $(X_i,A_i) = (\mathbb{C}P^{n_i},\mathbb{C}P^{m_i})$ with  $1\leq m_i <n_i\leq \infty$ (see \cite[Corollary 4.8]{St24a}).
\end{example}

\subsection*{Loop space decompositions}
We next prove certain coarse decompositions of polyhedral products of the form $\Omega \caa^\K$. To do this, we require the following result.

\begin{lemma}
\label{lem:susimpliesMoore}
    Let $n\geq 1$ and $X$ be a space such that $\Sigma^n X \in \bigvee(\mathcal{W} \cup \mathcal{M}_0)$, then for any prime $p$ and $r \geq 1$ such that $p^r \neq 2$, $P^{n+2}(p^r) \wedge X \in \bigvee\mathcal{M}_0$.
\end{lemma}
\begin{proof}
    By assumption, $\Sigma ^n X \simeq \bigvee_{j \in \mathcal{J}} S^{n_j}\vee \bigvee_{k \in \mathcal{K}}P^{n_k}(p_k^{r_k})$ for some indexing sets $\mathcal{J}$ and $\mathcal{K}$, where each $n_j \geq 2$, $n_k \geq 3$, and $p_k^{r_k} \neq 2$. There are homotopy equivalences \[P^{n+2}(p^r) \wedge X \simeq P^2(p^r) \wedge \Sigma^n X \simeq P^2(p^r)\wedge \left(\bigvee\limits_{j \in \mathcal{J}} S^{n_j} \vee \bigvee\limits_{k \in \mathcal{K}}P^{n_k}(p_k^{r_k})\right).\] Distributing the wedge sum over the smash product, there is a homotopy equivalence \begin{equation}\label{eqn:smashofMoore}P^2(p^r)\wedge \left(\bigvee\limits_{j \in \mathcal{J}} S^{n_j} \vee \bigvee\limits_{k \in \mathcal{K}}P^{n_k}(p_k^{r_k})\right) \simeq \bigvee\limits_{j \in \mathcal{J}} P^{n_j+2}(p^r) \vee \bigvee\limits_{k \in \mathcal{K}} (P^{2}(p^r) \wedge P^{n_k}(p_k^{r_k})).\end{equation}

    If $q$ and $q'$ are primes and $s,s' \geq 1$ such that $q^s \neq 2$, it was shown in  \cite[Corollary 6.6]{Nei80} that $P^{m}(q^s) \wedge P^{m'}(q^{s'})$ is contractible if $q \neq q'$, and if $q = q'$, there is a homotopy equivalence \[P^n(q^s) \wedge P^m(q^{s'}) \simeq P^{n+m}(q^{\min\{s,s'\}}) \vee P^{n+m-1}(q^{\min\{s,s'\}}).\]  Applying this to \eqref{eqn:smashofMoore}, $\bigvee_{j \in \mathcal{J}} P^{n_j+2}(p^r) \vee \bigvee_{k \in \mathcal{K}} (P^{2}(p^r) \wedge P^{n_k}(p_k^{r_k})) \in \bigvee(\mathcal{W} \cup \mathcal{M}_0)$ as claimed.
\end{proof}

\begin{theorem}
\label{thm:caatorsinPifzkinP}
Let $\K$ be a simplicial complex on the vertex set $[m]$ and let $A_1,\dots,A_m$ be spaces such that $\Sigma A_i \in \bigvee(\mathcal{W} \cup \mathcal{M}_0)$. If $\OZK \in \prod(\mathcal{P} \cup \mathcal{T}_0)$, then localised away from primes in $H_1(A_i)$ for all $i$, $\Omega \caa^\K \in \prod(\mathcal{P} \cup \mathcal{T}_0)$.
\end{theorem}
\begin{proof}
If each $H_1(A_i)$ is torsion free, we work integrally, otherwise localise away from primes appearing in $H_1(A_i)$ for all $i$. 

We proceed by induction on $l$, where $l$ is the number of $A_i$'s which are not homotopy equivalent to $S^1$. If $l=0$, each $A_i \simeq S^1$. Therefore, $\caa^\K \simeq \ZK$ and the result follows by assumption.

 Suppose the result is true for $l-1$ and consider the case $l \geq 1$. Pick $i$ such that $A_i \not \simeq S^1$. Let $Y_i = S^1$ and $Y_j = A_j$ for all $j \neq i$. By Lemma ~\ref{lmm:replacing A with Y}, it suffices to show $\Omega \caa^{\K \setminus i} \in \prod(\mathcal{P} \cup \mathcal{T}_0)$ and $\Omega\Sigma(F \wedge A_i) \in \prod(\mathcal{P} \cup \mathcal{T}_0)$. 

 By the inductive hypothesis, $\Omega \cyy^\K \in \prod(\mathcal{P} \cup \mathcal{T}_0)$ and $\Omega \caa^{\K \setminus i} \in \prod (\mathcal{P} \cup \mathcal{T}_0)$. Lemma~\ref{lmm:replacing A with Y} also implies that $\Omega \Sigma^2 F$ retracts off $\Omega \cyy^\K$. Hence, $\Omega \Sigma^2 F \in \prod (\mathcal{P} \cup \mathcal{T}_0)$. By part (4) of Lemma ~\ref{lem:relbetweenPandW}, $\Sigma \Omega \Sigma^2F \in \bigvee(\mathcal{W} \cup \mathcal{M}_0)$, and so the retraction of $\Sigma^2 F$ off $\Sigma \Omega \Sigma^2F$ implies by part (4) of Lemma ~\ref{lem:PTclosedunderret} that $\Sigma^2F \in \bigvee(\mathcal{W} \cup \mathcal{M}_0)$. Now consider $\Omega \Sigma(F \wedge A_i)$. Since $\Sigma A_i \in \bigvee(\mathcal{W} \cup \mathcal{M}_0)$, by shifting the suspension coordinate, there is a homotopy equivalence \[\Omega \Sigma(F \wedge A_i) \simeq \Omega \left(\bigvee\limits_{j \in \mathcal{J}} \Sigma^{n_j}F \vee \bigvee\limits_{j' \in \mathcal{J}'} P^{n_{j'}}(p^r) \wedge F\right)\]
 for indexing sets $\mathcal{J}$ and $\mathcal{J}'$. Note that each $n_j \geq 2$, and since we have localised away from primes appearing as $p$-torsion in $H_1(A_i)$ if necessary, each $n_{j'} \geq 4$. We have shown that $\Sigma^2F \in \bigvee(\mathcal{W} \cup \mathcal{M}_0)$, and so $\Sigma^{n_j} F \in \bigvee(\mathcal{W} \cup \mathcal{M}_0)$ for all $j$, and Lemma ~\ref{lem:susimpliesMoore} implies that $P^{n_{j'}}(p^r) \wedge F \in \bigvee\mathcal{M}_0$ for all $j'$. Hence, part (3) of Lemma ~\ref{lem:relbetweenPandW} implies that \[\Omega \left(\bigvee_{j \in \mathcal{J}} \Sigma^{n_j}F \vee \bigvee_{j' \in \mathcal{J}'} P^{n_{j'}}(p^r) \wedge F\right) \in \prod(\mathcal{P} \cup \mathcal{T}_0) \] as required.
\end{proof}

\begin{remark}
    In \cite{St25}, collections $\mathcal{M}$ and $\mathcal{T}$ were defined which contain $\mathcal{M}_0$ and $\mathcal{T}_0$ respectively, along with indecomposable $2$-torsion spaces. It is likely that an analogous result of Theorem ~\ref{thm:caatorsinPifzkinP} holds for the collections $\bigvee(\mathcal{W} \cup \mathcal{M})$ and $\prod(\mathcal{P} \cup \mathcal{T})$, however, the homotopy theory of these mod $2$ indecomposable spaces is more subtle. As our focus is on Anick's conjecture and we can localise away from $2$, we do not consider these here.
\end{remark}

Using parts (1) and (2) of Lemma ~\ref{lem:relbetweenPandW} and parts (1) and (3) of Lemma ~\ref{lem:PTclosedunderret}, the proof of Theorem ~\ref{thm:caatorsinPifzkinP} gives the following.

\begin{theorem}
    \label{thm:caainPifzkinP}
    Let $\K$ be a simplicial complex on $[m]$ and let $A_1,\dots,A_m$ be spaces such that $\Sigma A_i \in \bigvee\mathcal{W}$. Suppose $\OZK \in \prod\mathcal{P}$, then $\Omega \caa^\K \in \prod\mathcal{P}$. \qed
\end{theorem}

\begin{example}
    When $\K$ is a triangulation of a sphere, $\ZK$ has the structure of a manifold \cite[Theorem 4.1.4]{BP15}. Using homotopy theoretic techniques specific to manifolds, it was shown in \cite{ST24} that $\OZK \in \prod \mathcal{P}$ if $\K$ is a triangulation of $S^3$ or a neighbourly triangulation of an odd dimensional sphere, with an explicit enumeration of the spheres given in \cite{ST25}. 
    Theorem ~\ref{thm:caainPifzkinP} shows that $\Omega \caa^\K \in \prod\mathcal{P}$ if each $\Sigma A_i \in \bigvee \mathcal{W}$ in these cases, despite the fact that $\caa^\K$ often will not be a manifold.
\end{example}

We now turn our attention to Anick's conjecture. Recall that Anick's conjecture asserts that if $X$ is a simply connected, finite $CW$-complex, then localised at all but finitely many primes, $\Omega X \in \prod(\mathcal{P} \cup \mathcal{T}_0)$. Using Theorem ~\ref{thm:anick for djk}, we can verify Anick's conjecture for a wide range of polyhedral products. Note that many results also hold without the finiteness assumption.

\begin{theorem}
    Let $\K$ be a simplicial complex on the vertex set $[m]$. Let $A_1,\dots,A_m$ be spaces such that $\Sigma A_i \in \bigvee(\mathcal{W} \cup \mathcal{M}_0)$, and let $P$ be the set of primes $\tau(\OZK)\cup\{p:p<2m\}$ and the primes appearing as $p$-torsion in $H_1(A_i)$ for all $i$. Localised away from $P$, $\Omega \caa^\K \in \prod(\mathcal{P} \cup \mathcal{T}_0)$.
\end{theorem}
\begin{proof}
Localised away from $P$, Theorem ~\ref{thm:anick for djk} implies that $\Omega \ZK \in \prod\mathcal{P}$. Theorem ~\ref{thm:caatorsinPifzkinP} then implies $\Omega \caa^\K \in \prod(\mathcal{P} \cup \mathcal{T}_0)$.
\end{proof}

Using a result of Huang and Theriault, we show that $\caa^\K$ satisfies Anick's conjecture if $A_i$ are connected finite $CW$-complexes.
\begin{corollary}
\label{cor:finitecaa}
    Suppose that $A_1,\dots,A_m$ are connected finite $CW$-complexes where $A_i$ is $s_i$-connected and $d_i$-dimensional. Let $P$ be the set of primes in $\tau(\Omega \ZK)$, $p \leq \max{\{(d_i-s_i+1)/2,2m\}}$ for each $i$ and $\tau(A_i)$
    for each $i$. Then localised away from $P$, $\Omega\caa^\K\in\prod\mathcal{P}$. In particular, Anick's conjecture holds for $\caa^\K$.
\end{corollary}
\begin{proof}
By \cite[Lemma 5.1]{huang_theriault}, localised away from the set of primes $\tau(A_i)\cup\{p:p \leq (d_i-s_i+1)/2\}$, $\Sigma A_i \in \bigvee\mathcal{W}$. Hence $\Sigma A_i\in\bigvee\mathcal{W}$ localised away from $P$, so the result follows by Theorem ~\ref{thm:anick for djk} and Theorem ~\ref{thm:caainPifzkinP}.
\end{proof}

We can also give a characterisation for when Anick's conjecture holds for polyhedral products of the form $\ux^\K$.

\begin{corollary}
\label{cor:uxinP}
    Let $\K$ be a simplicial complex on $[m]$ and $X_1,\dots,X_m$ be simply connected $CW$-complexes. Then $\ux^{\K}$ satisfies Anick's conjecture if and only if all the spaces $X_i$ do.  
\end{corollary}
\begin{proof}
     Suppose there exists $i$ such that $\Omega X_i \notin \prod(\mathcal{P} \cup \mathcal{T}_0)$, but $\Omega \ux^\K \in \prod(\mathcal{P} \cup \mathcal{T}_0)$. By Theorem ~\ref{thm:uxasplitfib}, $\Omega X_i$ retracts off $\Omega \ux^\K$. Part (2) of Lemma ~\ref{lem:PTclosedunderret} implies $\prod(\mathcal{P} \cup \mathcal{T}_0)$ is closed under retracts, and so $\Omega X_i \in \prod(\mathcal{P} \cup \mathcal{T}_0)$ which is a contradiction.

    By Theorem ~\ref{thm:uxasplitfib}, there is a homotopy equivalence \[\Omega \ux^{\K} \simeq \prod\limits_{i=1}^m \Omega X_i \times \Omega (\underline{C \Omega X},\underline{\Omega X})^\K.\] If each $X_i$ satisfies Anick's conjecture, to show $\ux^\K$ satisfies Anick's conjecture, it suffices to show that $(\underline{C \Omega X},\underline{\Omega X})^\K$ does. By Theorem ~\ref{thm:anick for djk}, localised at a sufficiently large prime, $\Omega \ZK \in \prod\mathcal{P}$, and if each $X_i$ satisfies Anick's conjecture, localised at a sufficiently large prime, $\Omega X_i \in \prod(\mathcal{P} \cup \mathcal{T}_0)$. Part (4) of Lemma ~\ref{lem:relbetweenPandW} implies that $\Sigma \Omega X_i \in \bigvee(\mathcal{W} \cup \mathcal{M}_0)$. Hence, Theorem ~\ref{thm:caatorsinPifzkinP} implies that $\Omega (\underline{C \Omega X},\underline{\Omega X})^\K \in \prod(\mathcal{P} \cup \mathcal{T}_0)$.
\end{proof}

\begin{remark}
In \cite{anick_loop}, Anick states that the $p$-local version of Theorem ~\ref{thm:R-local BHS} holds for arbitrary subcomplexes of products of spheres, provided they have the standard $CW$-structure. Such a subcomplex is precisely a polyhedral product of the form $X=(\underline{S},\underline{\ast})^\K,$ where $S_1,\dots,S_m$ are spheres. Our results can be seen as an explanation and a far-reaching generalisation of Anick's remark.
\end{remark}

\subsection*{Poincar\'e series of loop homology}

Our final application is to calculate the Poincar\'e series of the loop homology of simply connected polyhedral products of the form $\uxa^\K$. The following two results give sufficient conditions for a polyhedral product to be simply connected.
\begin{lemma}
\label{lmm:caa^K simply connected}
Let $A_1,\dots,A_m$ be $CW$-complexes and $\K$ be a simplicial complex on $[m]$. Assume that all $A_i$ are connected, or $\K$ is $1$-neighbourly. Then $\caa^\K$ is simply connected.
\end{lemma}
\begin{proof}
A proof of the case where each $A_i$ is connected can be found in \cite[Proposition 5.3]{fat_wedge}.

Let $\K$ be $1$-neighbourly. Denote by $\sk_n X$ the $n$-skeleton of a $CW$-complex $X$.
We consider the $CW$-complexes $CA_i$ as reduced cones, so it follows that $\sk_0CA_i=\sk_0 A_i$. 

We claim $\sk_2\caa^\K=\sk_2\prod_{i=1}^m CA_i$. Indeed, any $2$-cell of $\prod_{i=1}^m CA_i$ is either a product of a single $2$-cell and some $0$-cells, which then belongs to $CA_i\times\prod_{j\neq i}A_j\subseteq\caa^\K$, or two $1$-cells and some $0$-cells, which then belongs to $CA_{i_1}\times CA_{i_2}\times\prod_{j\neq i_1,i_2}A_j\subseteq\caa^\K$ since $\{i_1,i_2\}\in\K$. Since $\pi_1(X)=\pi_1(\sk_2 X),$ it follows that $\pi_1(\caa^\K)\cong\pi_1(\prod_{i=1}^m CA_i)$, which is trivial.
\end{proof}
\begin{corollary}
\label{crl:uxa^K simply connected}
Let $(X_1,A_1),\dots,(X_m,A_m)$ be $CW$-pairs and $\K$ be a simplicial complex on $[m]$. Denote by $G_i$ the homotopy fibre of $A_i\to X_i$. Assume that each $X_i$ is simply connected. Also, assume all $G_i$ are connected or $\K$ is $1$-neighbourly. Then $\uxa^\K$ is simply connected.
\end{corollary}
\begin{proof}
By Theorem \ref{thm:gendecompuxa}, it is sufficient to show that $\prod_{i=1}^m X_i$ and $\cgg^\K$ are simply connected. The first statement holds by assumption, and the second holds by Lemma \ref{lmm:caa^K simply connected}.
\end{proof}

We first calculate the Poincar\'e series of the loop homology of polyhedral products of the form $\caa^\K$. Recall that $F(V;t):=\sum_{n\geq 0}\dim_\k(V_n)\cdot t^n\in\ZZ[[t]]$ is the Poincar\'e series of a graded $\k$-vector space $V$ of finite type.

\begin{theorem}
\label{thm:poincare series for hocaak}
    Let $A_1,\dots,A_m$ be spaces of finite type and $\K$ be a simplicial complex on $[m]$ such that $\caa^\K$ is simply connected. Let $\k$ be a field. Then $\Omega \caa^\K$ is of finite type, and $$1/F(H_*(\Omega\caa^\K;\k);t)=\sum_{J\subseteq[m]}\widehat{bb}_{\K_J,\k}(t)\cdot\prod_{j\in J}F(\H_*(A_j;\k);t)\in\ZZ[[t]].$$
\end{theorem}

\begin{proof}
We proceed by induction on $l$, where $l$ is the number of $A_i$'s which are not homotopy equivalent to $S^1$. If $l=0$, then $\caa^\K\simeq \ZK$, and $1/F(H_*(\OZK;\k);t)=\sum_J\widehat{bb}_{\K_J,\k}(t)t^{|J|}$ holds by Proposition \ref{prp:bb polynomials and loop homology of zk}. Also, $\OZK$ is of finite type since $\ZK$ is simply connected and of finite type.
 
Suppose the result is true for $l-1$ and consider the case $l \geq 1$. Pick $i$ such that $A_i \not \simeq S^1$. By Lemma ~\ref{lmm:replacing A with Y}, to show $\Omega \caa^\K$ is of finite type, it suffices to show that $\Omega \caa^{\K \setminus i}$ and $\Omega \Sigma(F \wedge A_i)$ are of finite type. Lemma ~\ref{lmm:replacing A with Y} implies $\Omega\caa^{\K \setminus i}$ and $\Omega \Sigma^2 F$ retract off $\Omega \cyy^\K$, which is of finite type by the inductive hypothesis. The retraction of $\Sigma^2 F$ off $\Sigma \Omega \Sigma^2 F$ implies $F$ is of finite type, and so the reduced K\"unneth theorem implies $\Sigma(F \wedge A_i)$ is of finite type. Hence, so is $\Omega \Sigma(F \wedge A_i)$ by the Bott-Samelson theorem.

Since the spaces $\Omega\caa^\K$, $\Omega\caa^{\K\setminus i}$, $\Omega\cyy^\K$, $F$, $A_i$ are of finite type, Poincar\'e series of their homology are well defined. In the notation of Lemma~\ref{lmm:replacing A with Y}, denote
    \begin{align*}
    x(t)&:=1/F(H_*(\Omega\caa^\K;\k);t), & y(t)&:=1/F(H_*(\Omega\cyy^\K;\k);t),\\
    z(t)&:=1/F(H_*(\Omega\caa^{\K\setminus i};\k);t), & w(t)&:=F(\H_*(F;\k);t).
    \end{align*}
    We will omit $\k$ and $t$ to simplify notation. Lemma ~\ref{lmm:replacing A with Y}, the (reduced) K\"unneth theorem and the Bott-Samelson theorem imply
    \begin{gather*}x=z\cdot (1-w\cdot F(\H_*(A_i))),\quad y=z\cdot (1-w\cdot F(\H_*(S^1))),\\
    \text{ hence }x-z=-z\cdot w\cdot F(\H_*(A_i))=(y-z)\cdot F(\H_*(A_i))/F(\H_*(S^1)).
    \end{gather*}
    
    By the inductive hypothesis applied to $\caa^{\K\setminus i}$ and $\cyy^\K$,
    \begin{gather*}
    z=\sum_{\begin{smallmatrix}J\subseteq[m]:\\
    i\notin J
    \end{smallmatrix}}\widehat{bb}_{\K_J}\cdot\prod_{j\in J}F(\H_*(A_j)),\quad
    (y-z)=\sum_{\begin{smallmatrix}J\subseteq[m]:\\
    i\in J
    \end{smallmatrix}}\widehat{bb}_{\K_J}\cdot F(\H_*(S^1))\prod_{j\in J\setminus i}F(\H_*(A_j)),\\
    \text{hence }x-z=\sum_{\begin{smallmatrix}J\subseteq[m]:\\
    i\in J
    \end{smallmatrix}}\widehat{bb}_{\K_J}\cdot F(\H_*(A_i))\prod_{j\in J\setminus i}F(\H_*(A_j))=
    \sum_{\begin{smallmatrix}J\subseteq[m]:\\
    i\in J
    \end{smallmatrix}}\widehat{bb}_{\K_J}\cdot\prod_{j\in J}F(\H_*(A_j))
    \end{gather*}
    and $x=\sum_{J\subseteq[m]}\widehat{bb}_{\K_J}\cdot\prod_{j\in J}F(\H_*(A_j)),$
    as required.
\end{proof}

\begin{theorem}
\label{thm:PSforuxa}
    Let $(X_1,A_1),\dots,(X_m,A_m)$ be $CW$-pairs and $\K$ be a simplicial complex on $[m]$ such that $\uxa^\K$ is simply connected and of finite type. Let $\k$ be a field. Denote by $G_i$ the homotopy fibre of $A_i \rightarrow X_i$. 
    Then
    $$\frac{1}{F(H_*(\Omega\uxa^\K;\k);t)}=\frac{\sum_{J\subseteq[m]}\widehat{bb}_{\K_J,\k}(t)\cdot\prod_{j\in J}F(\H_*(G_j;\k);t)}{\prod_{i=1}^m F(H_*(\Omega X_i;\k);t)}\in\ZZ[[t]].$$
\end{theorem}
\begin{proof}
    By Theorem ~\ref{thm:uxasplitfib}, there is a homotopy equivalence $\Omega \uxa^\K \simeq \Omega \cgg^\K\times\prod_{i=1}^m \Omega X_i.$  Since $\uxa^\K$ is simply connected of finite type, so is each $X_i$ and $\cgg^\K$. Moreover, $\Omega \uxa^\K$ is of finite type, implying that each $\Omega X_i$ and $\Omega \cgg^\K$ is also of finite type. The homotopy equivalence implies \[1/F(H_*(\Omega\uxa^\K;\k);t) = 1/F(H_*(\Omega \cgg^\K;\k);t)\prod\limits_{i=1}^m 1/F(H_*(\Omega X_i;\k);t).\] The result now follows from Theorem ~\ref{thm:poincare series for hocaak}.
\end{proof}

By Proposition \ref{prp:bb polynomials for flag complexes}, if $\K$ is a flag complex, then $\widehat{bb}_{\K_J,\k}(t) = 1-\chi(\K_J)$. Therefore, we obtain the following, which is a generalisation of a result of Cai \cite[Corollary 5.11]{Ca24} who considered polyhedral products of the form $(\underline{X},\underline{\ast})^\K$ where $\K$ is a flag complex.

\begin{corollary}
    Let $\K$ be a flag complex on $[m]$. Let $(X_1,A_1),\dots,(X_m,A_m)$ be $CW$-pairs, and $\k$ be a field. Denote by $G_i$ the homotopy fibre of $A_i \rightarrow X_i$. Assume that each $X_i$ is simply connected and all $G_i$ are connected. Then
    \[1/F(H_*(\Omega\uxa^\K;\k);t)=\sum_{J\subseteq[m]}(1-\chi(\K_J))\cdot\prod_{j\in J}F(\H_*(G_j;\k);t) \cdot \prod_{i=1}^m 1/F(H_*(\Omega X_i;\k);t)\in\ZZ[[t]].\eqno\qed
    \]
\end{corollary}

When $\Omega \uxa^\K \in \prod\mathcal{P}$, Theorem \ref{thm:PSforuxa} and Corollary \ref{crl:sphere-enumeration} allow to enumerate the terms which appear:
\begin{corollary}
\label{cor:enumsphere}
    Let $\K$ be a simplicial complex on $[m]$, and let $(X_1,A_1),\dots,(X_m,A_m)$ be $CW$-pairs. Denote by $G_i$ the homotopy fibre of $A_i \rightarrow X_i$. Suppose $\Omega \uxa^\K \in \prod\mathcal{P}$, that is there is a homotopy equivalence \[\Omega \uxa^\K \simeq (S^1)^{\times A} \times (S^3)^{\times B} \times (S^7)^{\times C}\times \prod_{n \geq 3,n \neq 4,8} (\Omega S^n)^{\times D_n}.\] Then the integers $A$, $B$, $C$ and $D_n$ are determined by the identity \[\frac{\sum_{J\subseteq[m]}\widehat{bb}_{\K_J,\QQ}(t)\cdot\prod_{j\in J}F(\H_*(G_j;\QQ);t) }{\prod_{i=1}^m F(H_*(\Omega X_i;\QQ);t)} = \frac{\prod_{n \geq 3, n \neq 4,8} (1-t^{n-1})^{D_n}}{(1+t)^A(1+t^3)^B(1+t^7)^C}.\qed\]
\end{corollary}
\begin{remark}
    An analogous result to Corollary ~\ref{cor:enumsphere} applies when $\Omega \uxa^{\K} \in \prod\mathcal{P}$ localised away from a set of primes. In this case, there may be extra terms of the form $S^{2n-1}$ with $n \geq 1$ which appear in the product decomposition. On the other hand, localised away from $2$, Serre's homotopy equivalence $\Omega S^{2n}\simeq S^{2n-1}\times\Omega S^{4n-1}$ allows one to assume that there are no terms of the form $\Omega S^{2n}$.   
\end{remark}
For polyhedral products of the form $\caa^\K$, the formula can be simplified as follows:
\begin{corollary}
\label{crl:caak-enum}
Let $\K$ be a simplicial complex on $[m]$, and let $A_1,\dots,A_m$ be pointed CW-complexes of finite type. Suppose that $\Omega\caa^\K\in \prod\mathcal{P}$. Then there is a homotopy equivalence
$$\Omega\caa^\K\simeq\prod_{n\geq 2}(\Omega S^n)^{\times D_n},$$
where the integers $D_n$ are determined by the identity
$$\sum_{J\subseteq[m]}\widehat{bb}_{\K_J,\QQ}(t)\cdot\prod_{j\in J}F(\H_*(A_j;\QQ);t)=\prod_{n\geq 2}(1-t^{n-1})^{D_n}.$$
In particular, in this case $\pi_N(\caa^\K)\simeq\bigoplus_{n=2}^N \pi_N(S^n)^{\oplus D_n}$ for all $N\geq 2$.

An analogous result holds if $\Omega\caa^\K\in\prod\mathcal{P}$ localised away from a set of primes.
\end{corollary}
\begin{proof}
By \cite[Lemma 6.1]{vylegzhanin}, $\Omega \caa^\K$ can be written as a finite type product of loops on simply connected spheres. The numbers $D_n$ are determined by computing Poincar\'e series twice, as in Corollary \ref{crl:sphere-enumeration} .
\end{proof}

Our final examples are real moment-angle complexes $\RK:=(D^1,S^0)^\K$ where $\K$ is $1$-neighbourly. Note that Theorem \ref{thm:caatorsinPifzkinP} does not apply to $\RK$, since $\Sigma S^0=S^1\notin\mathcal{W}=\{S^n:n\geq 2\}$. Hence, we do not know if Anick's conjecture holds for simply connected real moment-angle complexes. However, we can apply Theorem ~\ref{thm:PSforuxa} and Theorem ~\ref{thm:caaFPT} in the case $A_i=S^0$ to obtain the following.
\begin{theorem}
Let $\K$ be a $1$-neighbourly simplicial complex on $[m]$ and $\k$ be a field.
\begin{enumerate}
    \item 
    $1/F(H_*(\Omega\RK;\k);t)=\sum_{J\subseteq[m]}\widehat{bb}_{\K_J,\k}(t)\in\ZZ[t].$
    \item $\tau(\Omega\RK)\subseteq\tau(\OZK).$ In particular, $\Omega\RK$ has FPT. \qed
\end{enumerate}
\end{theorem}

\end{document}